\documentclass{amsart}
\usepackage[utf8]{inputenc}
\usepackage{amsmath,amsfonts,amssymb,amscd,amsthm,stmaryrd,xspace,mathtools}
\usepackage{hyperref}
\usepackage[foot]{amsaddr}
\usepackage[inline]{enumitem}

\theoremstyle{plain}
\newtheorem{theorem}{Theorem}[section]
\newtheorem*{theorem*}{Theorem}

\newtheorem{lemma}[theorem]{Lemma}
\newtheorem{proposition}[theorem]{Proposition}
\newtheorem{corollary}[theorem]{Corollary}
\newtheoremstyle{TheoremNum}
 {\topsep}{\topsep}              
 {\itshape}                      
 {}                              
 {\bfseries}                     
 {.}                             
 { }                             
 {\thmname{#1}\thmnote{ \bfseries #3}}
\theoremstyle{TheoremNum}

\theoremstyle{definition}
\newtheorem{example}[theorem]{Example}
\newtheorem{definition}[theorem]{Definition}

\theoremstyle{remark}
\newtheorem{remark}[theorem]{Remark}
\usepackage{tikz}
\usepackage{tikz-cd}
\usepackage[all]{xy}
\usetikzlibrary{calc}
\usepackage{bbm}
\usepackage{enumitem}
\usepackage{graphicx}
\usepackage{todonotes}
\numberwithin{equation}{section}
\setuptodonotes{noline, color=blue!30}
\usepackage{cleveref}
\title{Motivic Springer Theory}
\author{Jens Niklas Eberhardt}
\author{Catharina Stroppel}
\address{Mathematical Institute, University of Bonn, Endenicher Allee 60, 53115 Bonn, Germany}
\email{stroppel@math.uni-bonn.de, mail@jenseberhardt.com}

\newcommand{\A}{\mathbb{A}}
\newcommand{\F}{\mathbb{F}}
\newcommand{\Q}{\mathbb{Q}}
\newcommand{\Z}{\mathbb{Z}}

\newcommand{\Aa}{\mathcal{A}}
\newcommand{\Cc}{\mathcal{C}}
\newcommand{\Dd}{\mathcal{D}}
\usepackage[mathscr]{eucal}
\newcommand{\Ccc}{\mathcal{C}_{\infty}}
\newcommand{\Ddd}{\mathcal{D}_{\infty}}
\newcommand{\Rr}{\mathcal{R}}
\newcommand{\M}{\widetilde{\mathcal{N}}}
\newcommand{\N}{\mathcal{N}}
\newcommand{\Nnil}{\mathcal{N}_{nil}}

\newcommand{\Stein}{Z}

\newcommand{\Tt}{\mathcal{T}}
\newcommand{\un}{\Q}
\newcommand{\Chow}{\operatorname{CH}}
\newcommand{\Mot}{\operatorname{M}}

\newcommand{\D}{\operatorname{D}}
\newcommand{\h}{\operatorname{h}\!}
\newcommand{\Nerve}{\operatorname{N}\!}
\newcommand{\Ch}{\operatorname{Ch}\!}

\newcommand{\QQ}{\mathfrak{Q}}

\newcommand{\DM}{\operatorname{DM}}
\newcommand{\DTM}{\operatorname{DTM}}

\newcommand{\Dperf}{\operatorname{D_{perf}}}
\newcommand{\DperfZ}{\operatorname{D}^\Z_{\operatorname{perf}}}

\newcommand{\Hom}{\operatorname{Hom}}
\newcommand{\End}{\operatorname{End}}
\newcommand{\Ext}{\operatorname{Ext}}

\newcommand{\GL}{\operatorname{GL}}
\newcommand{\Fl}{\operatorname{Fl}}
\newcommand{\Gr}{\operatorname{Gr}}

\newcommand{\soc}{\operatorname{soc}}
\newcommand{\rad}{\operatorname{rad}}
\newcommand{\Aut}{\operatorname{Aut}}

\newcommand{\Comp}{\operatorname{Comp}}
\newcommand{\Compf}{\operatorname{Compf}}

\newcommand{\Fun}{\operatorname{Fun}}
\newcommand{\fin}{\operatorname{fin}}
\newcommand{\Rep}{\operatorname{Rep}}

\newcommand{\id}{\operatorname{id}}

\newcommand{\K}{\operatorname{K}}
\newcommand{\Ind}{\operatorname{Ind}}

\newcommand{\Res}{\operatorname{Res}}
\newcommand{\For}{\operatorname{For}}

\newcommand{\poi}{\operatorname{pt}}
\newcommand{\Spec}{\operatorname{Spec}}

\newcommand{\Dladic}{\operatorname{D}}
\newcommand{\Real}{\operatorname{Real}}

\newcommand{\DerG}{\operatorname{D}}

\newcommand{\disttrianglewithmaps}[5]{
	\begin{tikzcd}[column sep= 0.6cm]
		#1\arrow{r}{#2} \pgfmatrixnextcell #3 \arrow{r}{#4} \pgfmatrixnextcell#5\arrow{r}{+1}\pgfmatrixnextcell ~
	\end{tikzcd}
}
\newcommand{\disttriangle}[3]{\disttrianglewithmaps{#1}{}{#2}{}{#3}}

\usepackage{todonotes}
\begin{document}

\maketitle
\begin{abstract}
We show that representations of convolution algebras such as Lustzig's \emph{graded affine Hecke algebra} or the \emph{quiver Hecke algebra} and \emph{quiver Schur algebra} in type $A$ and $\widetilde{A}$ can be realised in terms of certain equivariant motivic sheaves called \emph{Springer motives}. To this end, we lay foundations to a \emph{motivic Springer theory} and prove formality results using weight structures. 

As byproduct, we express Koszul and Ringel duality in terms of a weight complex functor and show that partial quiver flag varieties in type $\widetilde{A}$ (with cyclic orientation) admit an affine paving.
\end{abstract}
\section{Introduction}
\subsection*{Motivation} An important insight in geometric representation theory is that algebras and their representations can often be constructed geometrically in terms of convolution of cycles. For example, the
\emph{Springer correspondence} \cite{springer_construction_1978} describes how irreducible representations of a Weyl group can be realised in terms of a convolution action on the free vector spaces spanned by irreducible components of \emph{Springer fibers}, see \cite{chriss_representation_2010}. Similar situations, which we refer to as \emph{Springer theories}, yield for example the \emph{affine Hecke algebra} \cite{lusztig_proof_1987}, the \emph{quiver Hecke algebra} (\emph{KLR algebra}) \cite{rouquier_2-kac-moody_2008} or the \emph{quiver Schur algebra} \cite{stroppel_quiver_2014} and their representations.

These constructions are usually employing Borel--Moore homology and constructible sheaves. Our main goal is to establish the foundations of a \emph{motivic Springer theory} using \emph{Chow groups} and \emph{motivic sheaves} instead. Motivic sheaves, see \cite{ayoub_les_2007} and \cite{cisinski_triangulated_2019}, are a relative version of Voevodsky's \emph{triangulated category of mixed motives} and their Hom-spaces are governed by \emph{Chow groups}. As established in the setting of flag varieties in \cite{soergel_perverse_2018}, motivic sheaves can serve as a \emph{graded version} of constructible sheaves that are technically advantageous over the mixed $\ell$-adic sheaves \cite{beilinson_faisceaux_1982} or mixed Hodge modules \cite{saito_young_2016}.

Convolution for Chow groups can be interpreted as composition for $\Hom$-spaces in categories of motivic sheaves. This motivates our definition of a \emph{motivic extension algebra}. We discuss how the graded affine Hecke algebra as well as quiver Hecke and quiver Schur algebras arise this way.
We then prove that \emph{purity} of certain fibers implies that the perfect derived category of a motivic extension algebra can be realised as a subcategory of equivariant motivic sheaves called \emph{Springer motives}---a statement we refer to as \emph{formality}.

To achieve our formality results, we make use of the theory of \emph{weight structures} from \cite{bondarko_weight_2010} (a concept independently introduced under the name \emph{co-$t$-structures} in \cite{pauksztelloCompactCorigidObjects2008} and studied already in the context of silting theory, see e.g.  \cite{Steffensilting})  and \emph{weight complex functors} from  \cite{bondarko_weight_2010}.
In a geometric context, formality is often obtained using of Deligne's \emph{yoga of weights}; for example one makes use of eigenvalues of a Frobenius morphism or weights of a mixed Hodge structure.
Various aspects of this yoga are formalised in the notion of  the \emph{Chow weight structure} on categories of motivic sheaves which we extend to the category of Springer motives.
In an algebraic context, we show that \emph{Koszul duality} and \emph{Ringel duality}, derived equivalences between  Koszul and Ringel dual algebras, respectively, can be expressed in terms of a weight complex functor.

\subsection*{Notation and conventions}
We use the term \emph{variety} for (not necessarily reduced) quasi-projective separated schemes of finite type over a field $k.$ In the introduction and most of the manuscript $k=\overline{\F}_p$ and $G$ is a linear algebraic group over $k.$ For rings $R\subset R'$ and $R$-modules $M$ we denote by $M_{R'}=M\otimes_RR'$ the extension of scalars. We use a cohomological convention for chain complexes and denote by $C[n]$ the cohomological shift with $(C[n])^i=C^{i+n}.$

\subsection*{Setup and main results} Let $\mu_i: \widetilde{\mathcal{N}}_i\to \N$ be a collection of $G$-equivariant proper maps of varieties such that each $\M_i$ is smooth. We consider the \emph{motivic extension algebra}
$$E=\bigoplus_{n\in \Z}\bigoplus_{i,j}\Hom_{\DM_G(\N,\Q)}(\mu_{i,!}(\Q_{\M_i}),\mu_{j,!}(\Q_{\M_j})(n)[2n]),$$
which is a motivic version of the $\ell$-adic extension algebra $E^{\acute{e}t}_\ell$ defined via the category $\D_G(\N,\Q_\ell)$ of equivariant $\ell$-adic sheaves. 

The algebra $E$ can be described in classical terms as the $G$-equivariant Chow groups of the Steinberg varieties $Z_{i,j}=\M_i\times_{\N}\M_j$ equipped with a convolution product. Namely, by Corollary ~\ref{cor:extensionalgebraaschowgroups} there is an isomorphism
$$
    E\cong\bigoplus_{i,j}\Chow_{\bullet}^G(Z_{i,j})_\Q.
$$
The motivic extension algebra is defined using the category $\DM_G(\N,\Q)$ of $G$-\emph{equivariant motivic sheaves} on $\N,$ which was introduced in \cite{soergel_equivariant_2018} as a motivic version of $\D_G(\N,\Q_\ell).$ Similarly to their $\ell$-adic counterpart, equivariant motivic sheaves are equipped with a six-functor-formalism. Their hom-spaces are governed by equivariant higher Chow groups and they admit an autoequivalence $(1)$ called  Tate twist, which will serve as a shift of grading functor for us.

We define the full subcategory of \emph{Springer motives}
$$\DM^{Spr}_G(\N,\Q)=\langle\mu_{i,!}(\Q_{\M_i})\rangle_{\cong,\inplus,\Delta,(\pm1)}\subset \DM_G(\N,\Q),$$
see Definition~\ref{def:springermotives}, and introduce local purity and finiteness conditions (PT) and (FO), see Section~\ref{sec:generalsetup}. We will prove the following \emph{formality} result showing that Springer motives can be described solely in terms of the algebra $E$ carrying \emph{no higher structure}.
\begin{theorem*}[Theorem~\ref{thm:main}]
Assuming (PT) and (FO) there is an equivalence of categories between the category of Springer motives and the perfect derived category of graded modules of the motivic extension algebra
\begin{center}
\begin{tikzcd}
\DM^{Spr}_G(\N,\Q) \arrow[r, "\sim"] & {\DperfZ(E)}.
\end{tikzcd}
\end{center}
Moreover, for all primes $\ell\neq p$ the $\ell$-adic realisation functor $\Real_\ell$ gives an isomorphism $E\otimes_\Q\Q_\ell\cong E^{\acute{e}t}_\ell$ and acts as a degrading functor with respect to the Tate-twist $(1)$ in the sense of \cite{beilinson_koszul_1996}
\begin{center}
\begin{tikzcd}
\DM^{Spr}_G(\N,\Q_\ell) \arrow[r, "\Real_\ell"]&
{\D^{Spr}_G(\N,\Q_\ell).}                                   
\end{tikzcd}
\end{center}
\end{theorem*}
\begin{remark}
The analogous statement replacing motivic sheaves by mixed $\ell$-adic sheaves (or mixed Hodge modules) fails because there are non-trivial extensions between the Tate objects $\Q_\ell(n)$ in the category of mixed $\ell$-adic sheaves. These unwanted extensions were first addressed in the context of perverse sheaves on flag varieties and category $\mathcal{O}$ in \cite[Section 4]{beilinson_koszul_1996}. There, a workaround is proposed using the derived category of mixed $\ell$-adic perverse sheaves which have a semisimple Frobenius action on their associated graded with respect to the weight filtration. This construction has several drawbacks when compared to motivic sheaves. First, it is not clear how to extend it to the equivariant case and to settings where no perverse $t$-structure exists. Secondly, it is difficult to determine if the six functors preserve the semisimplicity of the Frobenius. This and the independence of the prime $\ell$ are the main technical advantage of motivic sheaves.
\end{remark}
We apply the result in the setting of Lusztig's graded affine Hecke algebra $\overline{\mathbb{H}}(G)$ associated to the root datum of a reductive group $G,$ see also \cite{eberhardt_springer_2021}.
\begin{theorem*}[Theorem~\ref{thm:heckealgebra}]
Under a standard assumption on $p,$ there is an equivalence of categories
$$\DM^{Spr}_{G\times\mathbb{G}_m}(\Nnil,\Q)\cong \DperfZ(\overline{\mathbb{H}}(G))$$
between the category of Springer motives on the nilpotent cone $\Nnil$ and the perfect derived category of graded modules of $\overline{\mathbb{H}}(G).$
\end{theorem*}
We obtain a similar result for quiver Hecke algebras (KLR algebras) $R_\mathbf{d}$ for quivers $Q$ in type $A$ and $\widetilde{A}$ (with cyclic orientation).
\begin{theorem*}[Theorem~\ref{thm:quiverheckealgebra}]
There is an equivalence of categories
$$\DM^{Spr}_{\GL(\mathbf{d})}(\Rep(\mathbf{d}),\Q)\cong \DperfZ(R_\mathbf{d})$$
between the category of Springer motives for representations of $Q$ and the perfect derived category of graded modules of $R_\mathbf{d}.$
\end{theorem*}
A similar result holds for quiver Schur algebras, see Theorem~\ref{thm:quiverschuralgebra}. To establish the local purity condition (PT) in this setting, we prove that partial quiver flag varieties in type $\widetilde{A}$ (with cyclic orientation) admit affine pavings.
\subsection*{Summary} In Section~\ref{sec:weightstructuresandtilting}, we recall weight structures, weight complex functors and apply them to Koszul and Ringel duality. In Section~\ref{sec:formalismofequivariantmotivicsheaves}, we recall the formalism of (equivariant) motivic sheaves and establish important foundational results. In Section~\ref{sec:motivicspringertheory}, we introduce Springer motives, the motivic extension algebra and use the previous results to prove our main theorem on the formality of Springer motives. Finally, in Section~\ref{sec:heckealgebra} and~\ref{sec:quiverheckeandquiverschur} we discuss applications to affine Hecke algebras and quiver Hecke/Schur algebras.
\subsection*{Further directions}
\begin{enumerate*}

\item Most results should also hold with coefficients in characteristic $p$ by using an equivariant version of the formalism of motivic sheaves developed by Kelly and the first author in \cite{eberhardt_mixed_2019}, under a small restriction on the prime $p.$ 
\\
\item It should be possible to show that the category of $\ell$-adic Springer sheaves $\D^{Spr}_G(\N,\Q_\ell)$ is equivalent to the dg-derived category over the formal dg-algebra $(E^{\acute{e}t}_\ell,d=0)$ by combining our results and the techniques of \cite{schnurer_equivariant_2011}, where equivariant formality for the flag variety is discussed.
\end{enumerate*}
\subsection*{Relation to other work}
\begin{enumerate*}
%
%
\item 
General properties of geometric extension algebras and formality statements 
in the context of ($\ell$-adic) sheaves
have been discussed in the literature extensively, see for example \cite{sauter_survey_2013}, \cite{kato_algebraic_2017}, 
 \cite{mcnamara_representation_2020} and \cite{polishchuk_semiorthogonal_2019}.
 \\
 \item Formality results in the context of perverse sheaves can be found e.g. for graded affine Hecke algebras in \cite{rider_formality_2013},  \cite{rider_perverse_2016} and \cite{rider_formality_2021} and related to quiver Hecke algebras in \cite{mcnamara_monoidality_2017} and \cite{webster_weighted_2019}.
 \\






\end{enumerate*}


\subsection*{Acknowledgements} The first author thanks Hans Franzen for extensive discussions on quiver flag varieties. We thank Wolfgang Soergel and Matthias Wendt for their help with equivariant motivic sheaves.
This work was supported by the Hausdorff Center of Mathematics (grant EXC 2047).

\section{Weight structures and applications}\label{sec:weightstructuresandtilting}
We recall the definition and examples of weight structures due to Bondarko \cite{bondarko_weight_2010}. The main goal is the definition of the weight complex functor and applications to Koszul and Ringel duality.
\subsection{Weight structures}\label{sec:weightstructures}
We start with the definition of a weight structure, see \cite[Definition 1.1.1]{bondarko_weight_2010}.
\begin{definition}\label{def:weightstructure}
	Let $\Cc$ be a triangulated category. A \emph{weight structure} $w$ on $\Cc$ is a pair $w=(\Cc^{w\leq 0},\Cc^{w\geq 0})$ of idempotent-closed full subcategories of $\Cc,$ such that with $\Cc^{w\leq n}:=\Cc^{w\leq 0}[-n]$ and $\Cc^{w\geq n}:=\Cc^{w\geq 0}[-n]$ the following conditions are satisfied:
	\begin{enumerate}
		\item $\Cc^{w\leq 0}\subseteq \Cc^{w\leq 1}$ and $\Cc^{w\geq 1}\subseteq \Cc^{w\geq 0};$
		\item for all $X\in \Cc^{w\geq 0}$ and $Y\in\Cc^{w\leq -1}$, we have $\Hom_\Cc(X,Y)=0;$
		\item  \label{enum:triangleaxiom} for any $X\in \Cc$ there is a distinguished triangle 
		\begin{center}\disttriangle{A}{X}{B}\end{center} with $A\in \Cc^{w\geq 1}$ and $B\in \Cc^{w\leq 0}.$
	\end{enumerate}
	The full subcategory $\Cc^{w=0}=\Cc^{w\leq 0}\cap\Cc^{w\geq 0}$ is called the \emph{heart}. A weight structure is called \emph{bounded} if $\bigcup_{i} \Cc^{w\leq i} =\bigcup_{i} \Cc^{w\geq i}=\Cc.$ A triangulated functor $F:\Cc\to \Dd$ between two categories with weight structures is called \emph{weight exact} if $F(\Cc^{w\leq 0})\subset \Dd^{w\leq 0}$ and $F(\Cc^{w\geq 0})\subset \Dd^{w\geq 0}.$
\end{definition}
\begin{example}\label{exp:exampleweightstructures}
(1) The prototypical example of a weight structure arises from the stupid filtration of complexes. Let $\Aa$ be an idempotent-closed additive category. Then there is a bounded weight structure on $\K^b(\Aa)$ given by
\begin{align*}
    \K^b(\Aa)^{w\geq 0}&=\langle X \mid X_i = 0\text{ for all }i < 0\rangle_{\cong} \,\,\,\,\text{ and}\\
    \K^b(\Aa)^{w\leq 0}&=\langle X \mid X_i = 0\text{ for all }i > 0\rangle_{\cong},
\end{align*}
the subcategories generated by the complexes in non-negative and non-positive degrees under isomorphism. Most of the axioms of a weight structure are straigtforward to check. The idempotent-completeness $\K^b(\Aa)^{w\geq 0}$ and $\K^b(\Aa)^{w\leq 0}$ is discussed in \cite{schnurerHomotopyCategoriesIdempotent2011a}.
The heart of the weight structure is $\K^b(\Aa)^{w=0}=\Aa.$ In general, the heart of a weight structure is additive and idempotent closed but not necessarily abelian.
\\\noindent
(2) Assume that $\Aa$ is abelian and that every object in $\Aa$ has a finite projective resolution. Then one can identify the bounded derived category of $\Aa$ with the bounded homotopy category of projectives
$$\operatorname{D}^b(\mathcal{A})=\K^b(\operatorname{Proj}(\Aa)).$$
Hence $\operatorname{D}^b(\mathcal{A})$ is equipped with a weight structure with heart $\operatorname{Proj}(\Aa).$ This should be compared to the natural $t$-structure on $\operatorname{D}^b(\mathcal{A})$ with heart $\Aa.$
\\\noindent
(3) A particularly interesting example of weight structures arises in the world of motives, namely for Voevodsky's \emph{triangulated category of geometric motives}  $\DM_{gm}(k,\Q)$ over a perfect field $k$, see \cite{voevodsky_triangulated_2000}. The existence of a $t$-structure on $\DM_{gm}(k,\Q)$ is a notoriously difficult problem that implies, see \cite{beilinson_remarks_2010}, for example Grothendieck's standard conjectures. 

Instead of a $t$-structure, Bondarko \cite{bondarko_weight_2010} showed the existence of a weight structure $w$ called \emph{Chow weight structure} on the category $\DM_{gm}(k,\Q)$ whose heart
$$\DM_{gm}(k,\Q)^{w=0}\cong\operatorname{Chow}(k,\Q)$$
is equivalent to the category of \emph{Chow motives}. The category of Chow motives was introduced by Grothendieck and has an elementary definition, see \cite{milne_motives_2012}. Namely, one first considers the \emph{category of correspondences} of smooth projective varieties up to rational equivalence. Here, objects are smooth projective varieties $X$ over $k$ and morphisms from $X$ to $Y$ are elements in the rational Chow group
$$\Chow_{\dim(Y)}(X\times Y)_\Q.$$
Morphisms are composed via convolution. The category $\operatorname{Chow}(k,\Q)$ is then obtained from the category of correspondences by idempotent completion and tensor-inverting the \emph{Lefschetz motive} $\mathbb{L}=\ker(\mathbb{P}_k^1\to \Spec(k)).$
In other words, objects of weight zero in $\DM_{gm}(k,\Q)$ arise from motives of smooth projective varieties. 
\end{example}
Following \cite[Theorem 4.3.2]{bondarko_weight_2010} we now show how to define weight structures by specifying their heart.
\begin{definition}\label{def:tilting} A collection of objects $\Tt$ in a triangulated category $\Cc$ is called \emph{positive}, \emph{negative} or \emph{tilting}, respectively, if 
$$\Hom_\Cc(M,N[n])=0$$
for all $M,N\in\Tt$ where $n<0, n>0$ or $n\neq 0,$ respectively.
\end{definition}
\begin{proposition}\label{prop:defineweightstructure}
Let $\Cc$ be an idempotent closed triangulated category and $\Tt$ be a collection of objects in $\Cc.$ Assume that $\Tt$ is negative and generates $\Cc$ with respect to isomorphisms, direct summands and triangles
$$\Cc=\langle \Tt \rangle_{\cong, \inplus, \Delta}.$$

Then there is a bounded weight structure on $\Cc$ whose heart
$$\Cc^{w=0}=\langle \Tt \rangle_{\cong,\inplus,\oplus}$$
is the full subcategory of $\Cc$ generated by $\Tt$ under isomorphisms, direct summands and finite direct sums.
\end{proposition} 
\subsection{Weight complex functor} An object in a triangulated category $\Cc$ with bounded weight structure can be built from objects which are pure, that is in $\Cc^{w=n}=\Cc^{w=0}[-n]$ for $n\in \Z,$ using the distinguished triangles in Definition~~\ref{def:weightstructure}(~\ref{enum:triangleaxiom}), see \cite[Proposition 1.5.6]{bondarko_weight_2010}. This will imply that any geometric motive can be built from motives of smooth projective varieties. This should be seen as a reflection of Deligne's yoga of weights in the context of mixed Hodge structures and $\ell$-adic cohomology. 

We will now show that, most remarkably, these pure pieces can be assembled into a complex called \emph{weight complex}. If the category $\Cc$ admits some enhancement the weight complex gives a well-defined element in the homotopy category of $\Cc^{w=0}$ in a \emph{functorial way}.

We will use the following description of weight exact functors due to Sosnilo, see \cite[Proposition 3.3(b)]{sosnilo_theorem_2017}.
\begin{proposition}\label{prop:functorsbytheheart}
Let $\Ccc$ and $\Ddd$ be stable $\infty$-categories. Assume that their homotopy categories $\Cc=\h\Ccc$ and $\Dd=\h\Ddd$ are equipped with bounded weight structures. The hearts $\Ccc^{w=0}$ and $\Ddd^{w=0}$ are the full subcategories of $\Ccc$ and $\Ddd$ consisting of all objects in $\Cc^{w=0}$ and $\Dd^{w=0},$ respectively. Then restriction gives an equivalence of categories 
$$\Res:\Fun^{\operatorname{w-ex}}(\Ccc,\Ddd)\to\Fun^{\operatorname{add}}(\Ccc^{w=0},\Ddd^{w=0})$$
between the $\infty$-categories of exact functors from $\Ccc$ to $\Ddd$ that induce weight exact functors from $\Cc$ to $\Dd$ and of additive functors between the hearts $\Ccc^{w=0}$ and $\Ddd^{w=0}.$
\end{proposition}
The result can be used to construct a weight complex functor, see \cite[Corollary 3.5]{sosnilo_theorem_2017}:
\begin{proposition} \label{prop:weightcomplexfunctor} Let $\Cc$ be an idempotent complete triangulated category with bounded weight structure $w$. Assume that $\Cc=\h\Ccc$ arises as the homotopy category of a  stable $\infty$-category $\Ccc$. Then there is a functor of triangulated categories called \emph{weight complex functor}
\begin{align}
    t:\Cc\to\K^b(\Cc^{w=0})
\end{align}
that restricts to the natural embedding $\Cc^{w=0}\to \K^b(\Cc^{w=0})$ into degree $0.$
\end{proposition}
\begin{proof}
We apply Proposition~\ref{prop:functorsbytheheart} to $\Ccc$ and $\Ddd=\Nerve\Ch^b(\Cc^{w=0}),$ the stable $\infty$-category of bounded chain complexes with values in $\Cc^{w=0}$ from \cite[Example 4.4.5.1]{lurie_higher_2009} and \cite[Section 1.3.1]{lurie_higher_2012}.
The homotopy category  $\h\Ddd=\K^b(\Cc^{w=0})$ of $\Ddd$ is the bounded homotopy category of chain complexes in $\Cc^{w=0}$ and equipped with the canonical weight structure with heart $\Cc^{w=0},$ see Example~~\ref{exp:exampleweightstructures}(1). The heart of the $\infty$-category $\Ddd$ is $\Ddd^{w=0}=\Nerve(\Cc^{w=0}),$ the nerve of $\Cc^{w=0}.$ 

The unit of the adjunction between the nerve and homotopy category functors yields a functor $\epsilon:\Ccc^{w=0}\to\Ddd^{w=0}=\Nerve(\h\Ccc^{w=0}).$ The \emph{$\infty$-categorical weight complex functor}
$t_\infty: \Ccc\to \Ddd=\Nerve\Ch^b(\Cc^{w=0})$
is the unique (up to equivalence) weight exact functor such that $\Res(t_\infty)=\epsilon,$ where $\Res$ is the functor defined in Proposition ~~\ref{prop:functorsbytheheart}.
On the homotopy categories, $t_\infty$ induces the weight complex functor~$t$.
\end{proof}
\begin{remark}\label{rem:enhancements}
The weight complex functor also exists and admits an explicit construction if $\Cc$ admits an enhancement as an $f$-category, see \cite{bondarko_weight_2010} and \cite{schnurer_homotopy_2011}, or as a differential graded category, see  \cite{bondarko_differential_2009}. Moreover it also exists  if $\Cc$ admits an enhancement as a stable derivator using the fact that stable derivators yield $f$-categories, see \cite{modoi_reasonable_2019}.
\end{remark}
\begin{example} As an application we obtain a weight complex functor for the Chow weight structure
$$t:\DM_{gm}(k,\Q)\to\K^b(\operatorname{Chow}(k,\Q)),$$
see Example~\ref{exp:exampleweightstructures}(3). Here we use that $\DM_{gm}(k,\Q)$ has an enhancement as a stable $\infty$-category.
\end{example}
\begin{remark}
The weight complex functor for $\DM_{gm}(k,\Q)$ allows to decompose the motive $\Mot(X)$ of any (not necessarily smooth or projective) variety $X$ into a complex of motives of smooth projective varieties. Again, this reflects Deligne's yoga of weights for mixed Hodge structures and $\ell$-adic cohomology: the cohomology of a variety admits a weight filtration whose graded pieces behave like the cohomology of smooth projective varieties.
Similarly to Deligne's approach, the existence of the Chow weight structure and the weight complex functor relies on resolutions of singularities or de Jong's and Gabber's theory of alterations, see \cite{bondarko_1p-motivic_2011}.
\end{remark}
Sosnilo \cite[Corollary 3.5]{sosnilo_theorem_2017} and Aoki \cite[Theorem 4.3]{aoki_weight_2020} show that the weight complex functor is compatible with weight exact functors and also with symmetric monoidal structures as follows.
\begin{proposition}\label{prop:weightcomplexcommute}
Let $\Cc$ and $\Dd$ be triangulated categories with bounded weight structures and $F:\Cc\to\Dd$ be a weight exact functor. Assume that there is an enhancement to a functor $F: \Ccc\to \Ddd$ between stable $\infty$-categories. Then
\begin{enumerate}
    \item The following diagram of functors commutes up to natural isomorphism
\begin{center}
\begin{tikzcd}
\Cc \arrow[d, "t"] \arrow[r, "F"]    & \Dd \arrow[d, "t"] \\
\K^b(\Cc^{w=0}) \arrow[r, "\K^b(F)"] & \K^b(\Dd^{w=0}).  
\end{tikzcd}
\end{center}
\item  If $\Ccc$ is symmetric monoidal and $\Cc^{w\geq 0}$and $\Cc^{w\leq 0}$ are closed with respect to the monoidal structures then the weight complex functor can be turned into a symmetric monoidal functor.
\end{enumerate}
\end{proposition}

The weight complex functor reflects isomorphisms, but is not necessarily an equivalence. For example, the heart $\Cc^{w=0}$ is clearly tilting in $\K^b(\Cc^{w=0})$ while in general it is just negative in $\Cc.$
This is the main obstruction as the following result shows.
\begin{proposition}\label{prop:weightcomplexfunctorequivalence}
In the situation of Proposition~\ref{prop:weightcomplexfunctor}, assume that $\Cc^{w=0}$ is tilting in $\Cc$. Then the weight complex functor $t$ is an equivalence of categories.
\end{proposition}
\begin{proof}
Clearly, $t$ is fully faithful when restricted to $\Cc^{w=0}.$
Since $\Cc^{w=0}$ is tilting, $t$ is also fully faithful when restricted to $\bigcup_n\Cc^{w=n}.$ Now $\Cc^{w=0}$ generates $\Cc$ as triangulated subcategory since $w$ is bounded, see \cite[Corollary 1.5.7]{bondarko_weight_2010}. Hence $t$ is fully faithful on $\Cc$ by induction (dévissage) using the long exact sequence of $\Hom$-groups for distinguished triangles and the $5$-lemma. Essential surjectivity follows from dévissage as well since $\Cc^{w=0}$ generates $\K^b(\Cc^{w=0})$ as triangulated subcategory.
\end{proof}
\begin{corollary}\label{cor:triangulatedcatequalshomotopycat}
Let $\Cc$ be an idempotent-closed triangulated category that admits an enhancement as in Proposition~\ref{prop:weightcomplexfunctor} or Remark~\ref{rem:enhancements}. Let $\Tt$ be a collection of objects in $\Cc$ which is tilting. Then there is an equivalence of categories
$$t:\langle \Tt \rangle_{\cong, \inplus, \Delta}\to\K^b(\langle \Tt \rangle_{\cong, \inplus, \oplus})$$
between the full subcategory of $\Cc$ generated by $\Tt$ under isomorphisms, direct summands and triangles and the bounded homotopy category of the category generated  by $\Tt$ under isomorphisms, direct summands and finite direct sums.
\end{corollary}
\subsection{Locally unital algebras}\label{sec:locallyunital}
It is sometimes convenient to pass from additive categories, such as the heart of a weight structure $\Cc^{w=0},$ to categories of modules over some algebra with idempotents.
We use this perspective to prove a small variation on Corollary~\ref{cor:triangulatedcatequalshomotopycat}.
\begin{definition} A (non-unital) algebra $A$ with a distinguished set of idempotents $\{e_i|i\in I\}$ is called \emph{locally unital} if the canonical map
$$\bigoplus_{i,j\in I}e_iAe_j\to A$$
is an isomorphism. A right module over a locally unital algebra $A$ is called \emph{finitely generated (projective)} if it is isomorphic to a quotient (direct summand) of a finite direct sum of the modules $e_i A$ for $i\in I.$
The \emph{perfect derived category} of $A$ is the bounded homotopy category of the finitely generated projective right modules
$$\Dperf(A)=\K^b(\operatorname{mod_{fgp}-}A).$$

If $A$ is moreover $\Z$-graded we consider the category $\operatorname{mod}^\Z\operatorname{-}A$ of graded right modules over $A$ with morphisms of degree $0$ and the graded perfect derived category
$$\DperfZ(A)=\K^b(\operatorname{mod}^\Z_{\operatorname{fgp}}\operatorname{-}A).$$
\end{definition}
For a family of objects $\Tt$ in an idempotent-closed additive category $\Cc$ we consider the locally unital algebra with the distinguished idempotents $e_M=\id_M$ for $M\in \Tt$
\begin{align}
    \End_\Cc(\Tt)=\bigoplus_{M,N\in \Tt}\Hom_\Cc(M,N).\label{eqn:defendt}
\end{align}

By general nonsense we obtain:
\begin{proposition} The functor
$$\bigoplus_{M\in \Tt}\Hom_\Cc(M,-): \langle \Tt \rangle_{\cong, \inplus, \oplus} \to \operatorname{mod_{fgp}-}\End_\Cc(\Tt)$$
is an equivalence of categories.
\end{proposition}
 For a graded version, assume that $\Cc$ is equipped with an autoequivalence $\langle 1\rangle.$
Then one can define the $\Z$-graded locally unital algebra
\begin{align}\label{eq:gradedend}
\End_\Cc^\bullet(\Tt)=\bigoplus_{M,N\in \Tt} \Hom^\bullet_\Cc(M,N) \text{ where }\Hom^n_\Cc(M,N)=\Hom_\Cc(M,N\langle n\rangle).
\end{align}
Again, general nonsense yields
\begin{proposition}\label{prop:gradedcatasmodules} The functor
$$\bigoplus_{M\in \Tt}\Hom^\bullet_\Cc(M,-): \langle \Tt \rangle_{\cong, \inplus, \oplus,\langle \pm 1\rangle} \to \operatorname{mod}^\Z_{\operatorname{fgp}}\operatorname{-}\End_\Cc^\bullet(\Tt)$$
is an equivalence of categories.
\end{proposition}
\begin{corollary} \label{cor:triangulatedcatequalsmodulecat}
In the situation of Corollary~\ref{cor:triangulatedcatequalshomotopycat} the weight complex functor induces an equivalence
$$t:\langle \Tt \rangle_{\cong, \inplus, \Delta}\to \Dperf(\End_\Cc(\Tt)).$$
If $\Cc$ admits and autoequivalence $\langle 1\rangle$ such that $\widehat{\Tt}=\bigcup_n \Tt\langle n\rangle$ is also tilting then the weight complex functor induces an equivalence
$$t:\langle \widehat{\Tt} \rangle_{\cong, \inplus, \Delta}\to \DperfZ(\End^\bullet_\Cc(\Tt)).$$
\end{corollary}
\subsection{Koszul duality} Weight structures and weight complex functors can be used to provide a
convenient language for the \emph{Koszul duality} formalism from \cite{beilinson_koszul_1996}, and slightly more generally  \cite{mazorchuk_quadratic_2009}. 

Let $A$ be a $\Z$-graded algebra which is positively graded, that is, $A^i=0$ for $i<0$ and assume that $A_0$ is semisimple and finite dimensional. Denote by $\langle1\rangle$ the shift of grading functor on the category $A\text{-}\operatorname{mod}_{fg}^\Z$ of finitely generated graded $A$-modules and let $\Tt(A)$ be a set of representatives of simple objects concentrated in degree $0.$

The algebra $A$ is called \emph{Koszul} if $\Ext^i(L,L'\langle j\rangle)=0 \text{ for all } i\neq j$ and $L,L'\in \Tt(A).$ Equivalently, $A$ is Koszul if the family $\widehat{\Tt}(A)=\bigcup_i\Tt(A)\langle i\rangle[i]$ is tilting in $\DerG^b(A\text{-}\operatorname{mod}_{fg}^\Z).$ In this case, the \emph{Koszul dual} $A^!$ of $A$ is the graded algebra
$$A^!=\End^\bullet_{\DerG^b(A\text{-}\operatorname{mod}_{fg}^\Z)}(\Tt(\Aa))$$
where the right hand side is defined as in \eqref{eqn:defendt} using the autoequivalence $\langle 1\rangle[1].$ 

Corollary~\ref{cor:triangulatedcatequalsmodulecat} implies that the weight complex functor induces an equivalence of categories
\begin{align}
    t:\langle \widehat{\Tt}(A) \rangle_{\cong, \Delta}\stackrel{\sim}{\to} \Dperf(A^!).\label{eq:koszulweak}
\end{align}
In the case that $A$ is finitely generated as an $A_0$-left and right module and $A^!$ is left Noetherian, \eqref{eq:koszulweak} specializes to the \emph{Koszul duality} from \cite[Theorem 2.12.5]{beilinson_koszul_1996}
\begin{align}
    \DerG^b(A\text{-}\operatorname{mod}_{fg}^\Z)\stackrel{\sim}{\to} \DerG^b(A^!\text{-}\operatorname{mod}_{fg}^\Z).\label{eq:koszulstrong}
\end{align}
The heart of the weight structure defined by $\Tt(A)$ on the left hand side maps to projective modules in cohomological degree $0$ on the right. Moreover, one can show that the heart of the standard $t$-structure on the right hand side corresponds to the category of \emph{linear complexes} on the left hand side, see \cite{mazorchuk_quadratic_2009}.
\subsection{Ringel duality}
Similarly to the last paragraph, weight structures and weight complex functors can be applied to the theory of \emph{tilting objects} and \emph{Ringel duality} for highest
weight categories and their semi-infinite and stratified generalisation. For the general setup we refer to Brundan--Stroppel \cite{brundan_semi-infinite_2021} and freely use the terminology from there. 
Let $\Rr$ be a lower finite or essentially finite $\epsilon$-stratified category over an algebraically closed field (in particular it could be a highest weight category).

Then by \cite[Theorems 4.2, 4.13, 4.18]{brundan_semi-infinite_2021} \emph{tilting modules} in the sense of highest weight categories exist in $\Rr$ and form an additive subcategory generated by the family $\Tt(\Rr)$ of representatives of isomorphism classes of indecomposable tilting modules. 
For $E=\End_\Rr(\Tt(\Rr)),$ see \eqref{eqn:defendt}, the category $\Rr'=\operatorname{mod_{lfd}-}E$ of locally finite dimensional right modules is called the \emph{Ringel dual} of $\Rr.$ 

By \cite[Theorems 3.11, 3.56]{brundan_semi-infinite_2021} the family of tilting modules $\Tt(\Rr)$ is tilting in $\DerG^b(\Rr)$ in the sense of Definition~\ref{def:tilting}. Therefore Corollary~\ref{cor:triangulatedcatequalsmodulecat} implies that the weight complex functor induces an equivalence of categories
\begin{align}
    t:\langle \Tt(\Rr) \rangle_{\cong, \Delta}\stackrel{\sim}{\to} \Dperf(E).\label{eq:ringelweak}
\end{align}
Now assume that $\Tt(\Rr)$ generates $\DerG^b(\Rr)$ as a triangulated category and $E$ has finite cohomological dimension. For example, this is the case if $\Rr$ is a highest weight category. Then \eqref{eq:ringelweak} yields a derived equivalence, called \emph{Ringel duality}, between $\Rr$ and its Ringel dual $\Rr'$
$$\DerG^b(\Rr)\stackrel{\sim}{\to} \DerG^b(\Rr').$$
This interpretation of Ringel duality in terms of weight complex functors has interesting applications. For example, one can use Proposition~\ref{prop:weightcomplexcommute} to show that Ringel duality commutes with functors preserving tilting modules (under the correct technical assumptions).

\section{A formalism of equivariant motivic sheaves}\label{sec:formalismofequivariantmotivicsheaves}

Mixed $\ell$-adic sheaves or mixed Hodge modules are important tools in geometric representations theory. They are upgrades of the categories of $\ell$-adic sheaves and derived category of constructible sheaves on a complex variety with analytic topology, respectively. In particular, they are naturally equipped with a notion of \emph{weight filtration} and an endofunctor $(1),$ called \emph{Tate twist}, shifting this filtration. Arguments involving these weights can be very powerful.

In this section we will recall the formalism of (equivariant) motivic sheaves which has similar properties but two important technical advantages. First, the aforementioned weight filtration will be replaced by a grading. Secondly, all results work rationally and are hence indepedent of $\ell.$

\subsection{Recollections on motivic sheaves}
Let $k$ be a perfect field and $\poi=\Spec(k).$ All varieties are considered to be over $k.$

For a variety $X$ over $k$ we consider the triangulated category $\DM(X,\Q)$ of rational \emph{motivic sheaves on $X.$}\footnote{There are various definitions of $\DM(X,\Q)$ that are equivalent, see \cite[Section C.3]{cisinski_triangulated_2019}. We leave the choice of definition to the reader. In the literature, objects in $\DM(X,\Q)$ are often referred to as \emph{relative motives} or simply \emph{motives}. However, we prefer the term motivic sheaf and reserve the term motive for objects in $\DM(k,\Q)$.} 
In the special case $X=\Spec(k),$ the category $\DM(k,\Q)$ agrees with Voevodsky's \emph{triangulated category of mixed motives} over $k.$

The system of categories $\DM(X,\Q)$ has remarkable properties, some of which we will recall now. 

First, the work of Ayoub \cite{ayoub_les_2007,ayoub_les_20072} and Cisinski--Déglise \cite{cisinski_triangulated_2019} shows that $\DM(X,\Q)$ can be equipped with a six-functor-formalism which works very similarly as in the setting of $\ell$-adic sheaves. Important objects are the \emph{motive} $\Mot(X)$ and \emph{motive with compact support} $\Mot^c(X)$ of a variety $f:X\to \poi$ which can be expressed in terms of the six functors as
$$\Mot(X)=f_!f^!\Q\in \DM(k, \Q) \text{ and } \Mot^c(X)=f_*f^!\Q\in \DM(k, \Q).$$

There is an autoequivalence $(1)=-\otimes \Q(1)$ called \emph{Tate twist} on $\DM$ which commutes with the six functors and can be defined by splitting the motive of the projective line as
$$\Mot(\mathbb{P}^1_k)=\Q\oplus \Q(1)[2].$$

For each prime $\ell$ invertible in $k$, there is an $\ell$-adic \emph{realisation functor} to the category $\Dladic(X,\Q_\ell)$ of $\ell$-adic sheaves on $X$
\begin{align}
    \Real_\ell: \DM(X,\Q)\to \Dladic(X,\Q_\ell),\label{eq:realisationfunctor}
\end{align}
which is compatible with the six functors and the Tate twist, see \cite{ayoub_realisation_2014}.

Morphisms in $\DM(X,\Q)$ are best understood in terms of higher Chow groups, as defined by Bloch \cite{bloch_algebraic_1986}. For $X$ smooth, there is a natural isomorphism
\begin{align}
    \Hom_{\DM(X,\Q)}(\un_X,\un_X(m)[n])&\cong\Hom_{\DM(k,\Q)}(\Mot(X),\un(m)[n])\nonumber\\
    &\cong\Chow^m(X,2m-n)_\Q,\label{eqn:motiviccomologychowgroups}
\end{align}
where $\Q_X$ denotes the tensor unit and $\Chow$ a higher Chow group. In particular, for $n=2m$ one obtains the usual Chow group of codimension-$m$ cycles $\Chow^m(X,0)_\Q=\Chow^m(X)_\Q.$ In this case the realisation functor $\Real_\ell$ yields the cycle class map to $\ell$-adic cohomology
$$\Chow^m(X)_\Q\to H^{2m}_{\text{\'et}}(X,\Q_\ell(m)).$$
For $X=\poi$ and $k=\F_q$ or $k=\overline{\F}_p$ these hom-groups are particularly simple:
\begin{lemma}\label{lem:motiviccohomologyofpoint} Let $k=\F_q$ or $k=\overline{\F}_p$, then
\begin{align}\label{eqn:motiviccohomologyofpoint}
    \Hom_{\DM(k,\Q)}(\un,\un(m)[n])=\left\{
\begin{array}{cl}
\Q & \text{ for }n=m=0 \text{ and} \\
0 &  \text{otherwise.} \\
\end{array}
\right.
\end{align}
\end{lemma}
\begin{proof}
First, note that the higher $K$-theory of finite fields $K_i(\Spec(\mathbb{F}_q))$ is torsion for $i>0$ by \cite[Theorem 8]{quillen_cohomology_1972}. The same is true for the algebraic closure, since $K$-theory commutes with filtered colimits. Hence, $K_i(k)_\Q=0$ for $i>0.$
By the Riemann--Roch theorem for rational higher Chow groups, see \cite[Theorem 9.1]{bloch_algebraic_1986}, $\Chow^m(k,2m-n)_\Q$ is a direct summand of $K_{2m-n}(k)_\Q.$
The statement follows from  \eqref{eqn:motiviccomologychowgroups}.
\end{proof}
In the following the purity property in  \eqref{eqn:motiviccohomologyofpoint} will be crucial.
For this reason, we will from now on restrict to the case $k=\overline{\F}_p.$

\subsection{Mixed Tate motives} \label{sec:tatemotives}
We will now recall the definition of the category of pure/mixed Tate motives, see \cite{levine_mixed_2005}, which will play for us the role of a \emph{graded version} of the derived category of sheaves on the point.
\begin{definition}
The category of \emph{mixed Tate motives}\footnote{In the literature, there are many notations for the category of mixed Tate motives. For example, $\operatorname{TDM}$ in \cite{huber_slice_2006}, $\operatorname{MTDer}$  in \cite{soergel_perverse_2018,soergel_equivariant_2018}, $\operatorname{DMT}$ in \cite{spitzweck_notes_2016} or $\operatorname{DTM}$ in \cite{levine_mixed_2005}.} is the subcategory
$$\DTM(k,\Q)=\langle\un\rangle_{\cong, \inplus, \Delta, (\pm 1)}\subset\DM(k,\Q)$$
generated by $\un$ under isomorphism, direct summands, Tate twists and triangles. The category of \emph{pure Tate motives} is the category generated by objects $\Q(n)[2n]$
$$\DTM(k,\Q)^{w=0}=\langle \Q(n)[2n]\mid n\in \Z \rangle_{\cong,\inplus,\oplus}\subset \DTM(k,\Q)$$
with respect to isomorphism, direct summands and finite direct sums.
\end{definition}

For example, the decompostion of a projective space into affine spaces
$$\mathbb{P}^n=\A^0\uplus \A^1 \uplus\dots \uplus \A^n$$
induces a decomposition of its motive into a direct sum of Tate motives
$$\Mot(\mathbb{P}^n)=\Q\oplus \Q(1)[2]\oplus\dots\oplus \Q(n)[2n]$$
which shows that $\Mot(\mathbb{P}^n)$ is pure Tate.

We will use the following generalization.
\begin{definition}
A partition of a variety $X$ into subvarieties $X_1,\dots, X_n$ (called \emph{strata}) is an \emph{affine paving}, if $X
_{\leq k}=\bigcup_{i=1,\dots,k}X_i$
is closed in $X$ for all $1\leq l \leq n$ and each stratum $X_i$ is isomorphic to an affine space $\mathbb{A}^n.$
\end{definition}
\begin{proposition}\label{prop:affinepavingpuretate} Let $X$ be a variety that admits an affine paving. Then the motive with compact support
$\Mot^c(X)\in \DTM(k, \Q)^{w=0}$
is pure Tate.
\end{proposition}
\begin{proof}
This follows from the localisation sequence and $\Mot^c(\A^n)=\Q(n)[2n].$ A proof can be found in \cite[Lemma 2.3(10)]{eberhardt_springer_2021}.
\end{proof}
As the notation suggests, the category of pure Tate motives is the heart of a weight structure:
\begin{proposition}
The category $\DTM(k,\Q)^{w=0}$ is the heart of a weight structure $w$ on $\DTM(k,\Q).$
\end{proposition}
\begin{proof}
By~\eqref{eqn:motiviccomologychowgroups} the collection of objects in the $\DTM(k,\Q)^{w=0}$ is negative. Moreover $\DTM(k,\Q)$ is idempotent closed and generated by $\DTM(k,\Q)^{w=0}$. Now the  statement follows from Proposition~\ref{prop:defineweightstructure}.
\end{proof}
Using that $k=\overline{\mathbb{F}}_p$ we obtain a simple decription of the category of Tate motives:
\begin{proposition}
The weight complex functor yields an equivalence of categories
$$t:\DTM(k,\Q)\to\K^b(\DTM(k,\Q)^{w=0})\cong\D^b(\Q\operatorname{-mod}^\Z)$$
where $\Q(1)[2]$ corresponds to the one-dimensional vector space in degree one $\Q\langle 1\rangle$ in the category $\Q\operatorname{-mod}^\Z$ of graded finite-dimensional vector spaces over $\Q.$
\end{proposition}
\begin{proof}
Lemma~\ref{lem:motiviccohomologyofpoint} implies that $\DTM(k,\Q)^{w=0}$ is tilting in $\DTM(k,\Q).$ Further, $\DM$ admits an enhancement as stable $\infty$-category. The statement follows from Corollary~\ref{cor:triangulatedcatequalsmodulecat}.
\end{proof}

We note that the weight structure on mixed Tate motives considered here is just the shadow of the Chow weight structure on $\DM^{gm}(k,\Q),$ see Example~\ref{exp:exampleweightstructures}.

\subsection{Equivariant motivic sheaves}\label{sec:equivariantmotives}
Soergel--Virk--Wendt \cite{soergel_equivariant_2018} introduced an \emph{equivariant} version of the above formalism.
For a variety $X$ with an action of a linear algebraic group $G$
they define the category $\DM_G(X,\Q)$ of \emph{$G$-equivariant motivic sheaves} on $X.$ This system of categories still carries a six-functor-formalism\footnote{For technical reasons, Soergel--Virk--Wendt construct a six-functor-formalism for the full subcategory $\DM^+_G(X,\Q)\subset \DM_G(X,\Q)$ of objects that are bounded below with respect to the homotopy $t$-structure, see \cite[Section I.6]{soergel_equivariant_2018}. All objects that we consider here are automatically in  $\DM^+_G(X,\Q)$ and we simply ignore this technicality.}  and induction/restriction functors changing the group.

Similarly to the non-equivariant case, there is a realization functor
to the equivariant derived category of $\ell$-adic sheaves  $\Dladic_G(X,\Q_\ell)$ of Bernstein--Lunts \cite{bernstein_equivariant_1994}
$$\Real_\ell: \DM_G(X,\Q)\to \Dladic_G(X,\Q_\ell).$$
Moreover, for $X$ smooth there is a natural isomorphism
\begin{align}
    \Hom_{\DM_G(X,\Q)}(\un_X,\un_X(m)[n])&
    \cong\Chow^m_G(X,2m-n)_\Q\label{eqn:equivariantmotiviccomologychowgroups}
\end{align}
to the equivariant higher Chow groups as defined by Totaro \cite{raskind_chow_1999} and Edidin--Graham \cite{edidin_equivariant_1998}.

There is a forgetful functor from equivariant to non-equivariant motivic sheaves
$$\For: \DM_G(X,\Q)\to \DM(X,\Q)$$
commuting with the six functors.

An important property of equivariant (motivic) sheaves is the \emph{induction equivalence}, which allows to describe $G$-equivariant motivic sheaves on a $G$-orbit $G/H$ in terms of $H$-equivariant motives on a point:
\begin{proposition}
Let $i:H\hookrightarrow G$ be a closed subgroup. Denote by $s:X\to G\times_H X, x\mapsto [e,x].$ If the anti-diagonal action of $H$ on $G\times X$ is free then there is an equivalence of categories
\begin{align}
    (i,s)^*:\DM_G(G\times_H X)\stackrel{\sim}{\to}\DM_H(X).\label{eq:inductionequivalence}
\end{align}
\end{proposition}
\begin{proof}
See \cite[Proposition I.7.4]{soergel_equivariant_2018}.
\end{proof}
\begin{remark}
In \cite{soergel_equivariant_2018} the language of derivators is used and it is shown that $\DM_G(X,\Q)$ admits an enhancements as stable derivator. In \cite{richarz_intersection_2020}, \cite{richarz_motivic_2021} and \cite{richarz_tate_2021} Richarz--Scholbach provide a similar construction in the language of $\infty$-categories which shows that $\DM_G(X,\Q)$ admits an enhancement as stable $\infty$-category.
\end{remark}

\subsection{Equivariant mixed Tate motives} There is also an equivariant version of the category of mixed Tate motives defined in Section~\ref{sec:tatemotives}. For the remainder of the section let $k=\overline{\F}_p.$
\begin{definition}
The category of \emph{equivariant mixed Tate motives} on a  point $$\DTM_G(k,\Q)\subset\DM_G(k,\Q)$$
is the full subcategory of objects $M$ such that $\For(M)\in \DTM(k,\Q),$
\end{definition}
We now explain how the explicit description of this category in \cite[Theorem II.3.1]{soergel_equivariant_2018} can be obtained using the weight complex functor.
\begin{proposition}
There is a weight structure $w$ on $\DTM_G(k,\Q)$ such that
\begin{align*}
    M\in \DTM_G(k,\Q)^{w\geq 0} &\iff \For(M)\in \DTM(k,\Q)^{w\geq 0}\text{ and }\\
    M\in \DTM_G(k,\Q)^{w\leq 0} &\iff \For(M)\in \DTM(k,\Q)^{w\leq 0}.
\end{align*}
\end{proposition}
\begin{proof}
This is \cite[Proposition II.4.10]{soergel_equivariant_2018}.
\end{proof}
Denote by $G^0\subset G$ the connected component of the identity. Then one can consider the object $\Ind^G_{G^0}(\Q)\in \DM_G(k)$ which plays the role of the local system on $BG$ with free action of the component group $\pi_0(G)=G/G^0.$ The heart of $w$ admits the following explicit description by  \cite[Proposition II.4.5]{soergel_equivariant_2018}.
\begin{proposition}\label{prop:heartofweightstrutureequivariantmixedtate}
The heart of the weight structure on $\DTM_G(k,\Q)$ is generated by the objects $\Ind^G_{G^0}(\Q)(n)[2n]$ with respect to isomorphism, direct summands and finite direct sums
$$\DTM_G(k,\Q)^{w=0}=\langle \Ind^G_{G^0}(\Q)(n)[2n] \mid n \in \Z \rangle_{\cong,\inplus,\oplus}\subset \DTM_G(k,\Q).$$
\end{proposition}
For the moment, assume that $G=G_0$ is connected. By \cite[Proposition I.7.6]{soergel_equivariant_2018} the categories $\DM_G$ just depend on the quotient $G/R_u(G)$ of $G$ by its unipotent radical. Hence, we can assume that $G$ is reductive. Denote by $T$ a maximal torus in $G$ and by $W=N_G(T)/T$ the Weyl group. Let $S=\operatorname{Sym}(X(T)_\Q)$ be the symmetric algebra of the rationalized character lattice of $T.$ The algebra $S$ is isomorphic to a polynomial ring in $\operatorname{rank}(T)$ many variables and graded where we put $X(T)$ in degree one.

Since we are only considering rational coefficients, the $T$- and $G$-equivariant Chow rings agree with equivariant cohomology rings. In particular, the Chern class map induces isomorphisms of graded algebras
\begin{align}
    S\cong \Chow^\bullet_T(k)_\Q \text{ and }
    S^W\cong \Chow^\bullet_T(k)_\Q^W\cong\Chow^\bullet_G(k)_\Q\label{eq:equivariantchowofpoint},
\end{align}
see \cite{raskind_chow_1999} or \cite[Section 3.2]{edidin_equivariant_1998}. Similarly, for higher Chow groups
\begin{align}\label{eq:higherequivariantchowofpoint}
    \Chow^\bullet_G(k,i)_\Q\cong(\Chow^\bullet_T(k,i)_\Q)^W\cong (S\otimes\Chow^\bullet(k,i)_\Q)^W
\end{align}
using \cite[Theorem 1.5]{krishna_equivariant_2017} and \cite[Theorem 5.7]{krishna_higher_2013}. 

If $G$ is not connected, we consider the extension algebra
\begin{align*}
    E&=\bigoplus_{n\in \Z}\Hom_{\DM_G(k,\Q)}(\Ind^G_{G^0}(\Q),\Ind^G_{G^0}(\Q)(n)[2n])
\end{align*}
which can be thought of as the Chow ring of $BG$ with coefficients in the local system given by the regular representation of $\pi_0(G).$ For an explicit description denote by $S^W\ltimes \Q[\pi_0(G)]$ the \emph{twisted group algebra}, see \cite[A.2.3]{soergel_equivariant_2018}, which as a graded vector space is just the tensor product  $S^W\otimes \Q[\pi_0(G)]$ with $\Q[\pi_0(G)]$ concentrated in degree zero.
\begin{proposition}\label{prop:extofpoint}
There is a natural isomorphism 
$E\cong S^W\ltimes \Q[\pi_0(G)].$
\end{proposition}
\begin{proof}
If $G$ is connected this is \eqref{eq:equivariantchowofpoint}. Otherwise the statements follow from a transfer argument for finite group torsors, see \cite[Section A.2.2]{soergel_equivariant_2018}.
\end{proof}
\begin{proposition}\label{prop:tiltingpointequivariant}
The collection of objects $\{\Ind^G_{G^0}(\Q)(n)[2n]\mid n\in \Z\}$ is tilting.
\end{proposition}
\begin{proof}
Assume first that $G$ is connected. By \eqref{eq:equivariantchowofpoint} it suffices to show that the higher Chow groups in \eqref{eq:equivariantchowofpoint} vanish for $i\neq0.$ This holds by Lemma~~\ref{lem:motiviccohomologyofpoint} using  $k=\overline{\F}_p.$ Again, the statement for $G$ not connected follows from \cite[Section A.2.2]{soergel_equivariant_2018}.
\end{proof}

With Corollary~\ref{cor:triangulatedcatequalsmodulecat} we obtain the following explicit description of $\DTM_G(k,\Q).$
\begin{theorem}
The weight complex functor induces an equivalence
$$t: \DTM_G(k,\Q)\to\DperfZ(S^W\ltimes \Q[\pi_0(G)]).$$
\end{theorem}
Since $S^W\ltimes \Q[\pi_0(G)]$ has finite cohomological dimension,  $\DperfZ(S^W\ltimes \Q[\pi_0(G)])$ is just the bounded derived category of finitely generated graded modules. A similar result was shown in \cite[Theorem II.3.1]{soergel_equivariant_2018} using slightly different arguments.
\subsection{Gradings} Let $\ell\neq p$ be a prime. The categories of equivariant mixed Tate motives $\DTM_G(k)$ can be regarded as \emph{graded versions} of the categories of equivariant sheaves $\DerG_G(\poi,\Q_\ell)$ defined by  Bernstein--Lunts \cite{bernstein_equivariant_1994}.
Under the realisation functor $\Real_\ell,$ see \eqref{eq:realisationfunctor}, the Tate motive $\Q(1)$ gets mapped to the Tate module
$$\Real_\ell(\Q(1))=\Q_\ell(1)=\varprojlim \mu_{\ell^n}\otimes_{\Z_\ell}\Q_\ell$$
which can be identified with $\Q_\ell$ by choosing a compatible system of $\ell^n$-th roots of unity in $k=\overline{\F}_p.$ This induces a natural equivalence of functors
\begin{align}
    \Real_\ell\circ (1)\to\Real_\ell.\label{eq:realforgetsgrading}
\end{align}
Hence, intuitively, the Tate twist $(1)$ can be regarded as a shift of grading and $\Real_\ell$ as a functor forgetting the grading. Restricted to mixed Tate motives, the functor $\Real_\ell$ becomes a \emph{degrading functor} in the sense of \cite[Section 4.3]{beilinson_koszul_1996}.
\begin{proposition}\label{prop:puretateonpointgrading}
The equivalence in \eqref{eq:realforgetsgrading} induces an isomorphism
$$\bigoplus_{n\in\Z}\Hom_{\DM_G(k,\Q_\ell)}(M,N(n))\to \Hom_{\DerG_G(k,\Q_\ell)}(\Real_\ell(M),\Real_\ell(N))$$ for $M,N\in \DTM_G(k).$ 
If $M,N\in\DTM_G(k)^{w=0}$ then all summands for $n\neq 0$ vanish and the functor $\Real_\ell$ gives an isomorphism
$$\Hom_{\DM_G(k,\Q_\ell)}(M,N)\to \Hom_{\DerG_G(k,\Q_\ell)}(\Real_\ell(M),\Real_\ell(N)).$$
\end{proposition}
\begin{proof} First, by Proposition~\ref{prop:heartofweightstrutureequivariantmixedtate} and induction it suffices to show the statement for objects of the form $\Ind_{G_0}^G\Q(n)[m].$ If $G$ is connected
the homomorphims of objects of the form $\Ind_{G_0}^G\Q(n)[m]=\Q(n)[m]$ are described in terms of $S^W$ in both $\DM_G(k,\Q),$ see Proposition~\ref{prop:extofpoint}, and in $\DerG_G(\poi,\Q_\ell),$ see \cite[Section 13.10]{bernstein_equivariant_1994}, and the statement is easily seen to be true. The case that $G$ is not connected can be handled as described in \cite[Theorem A.2.8]{soergel_equivariant_2018}.
\end{proof}

\subsection{Pointwise Tate}\label{sec:pointwisetate}
We need a notion of local purity for motivic sheaves on a variety $X.$ That is, we want to consider motivic sheaves whose restriction to every point is pure or mixed Tate:
\begin{definition} Let $?\in \{*,!\}.$
An object $M\in \DM_G(X)$ is called \emph{$?$-pointwise mixed Tate} or \emph{$?$-pointwise pure Tate}, respectively, if for each point $i_x:\poi\to X$
\begin{align*}
    i_x^?\For M\in \DTM(k,\Q) \text{ or }
    i_x^?\For M\in \DTM(k,\Q)^{w=0} \text{, respectively.}
\end{align*}
Here the functor $i_x^?\For$ is the composition 
$$\DM_G(X,\Q)\stackrel{\For}{\to}\DM(X,\Q)\stackrel{i_x^?}{\to}\DM(k,\Q).$$ The object $M$ is called \emph{pointwise mixed (pure) Tate} if it is $*$- and $!$-pointwise mixed (pure) Tate.
\end{definition}
It is sometimes convenient to work with the following equivalent orbitwise definition.
\begin{proposition} Let $?\in\{*,!\}$ and $M\in \DM_G(X).$ Then $M$ is $?$-pointwise mixed or pure Tate if and only if for each orbit $\mathcal{O}\hookrightarrow X$
\begin{align*}
    (i,s)^*j^?M\in\DTM_H(k,\Q) \text{ or }
    (i,s)^*j^?M\in\DTM_H(k,\Q)^{w=0} \text{, respectively.}
\end{align*}
Here $j:G/H\cong\mathcal{O}\hookrightarrow X$ and the functor $(i,s)^*j^?$ is the composition 
$$\DM_G(X,\Q)\stackrel{j^?}{\to}\DM_G(G/H,\Q)\stackrel{(i,s)^*}{\to}\DM_H(\poi,\Q)$$ of pullback to the orbit and the induction equivalence \eqref{eq:inductionequivalence}.
\end{proposition}
\begin{remark} In \cite{soergel_equivariant_2018} this equivalent orbitwise definition is used.
\end{remark}
Objects that are pointwise pure Tate have remarkable properties. They behave very similarly to pure Tate objects on a point, particulary if there are only finitely many $G$-orbits. They satisfy the following extension vanishing:
\begin{proposition}\label{prop:homvanishingpure} Assume that the $G$-action on $X$ has finitely many orbits. Let $M,N\in \DM_G(X)$ be $*$- and $!$-pointwise pure Tate. Then $\Hom_{\DM_G(X)}(M,N[n])=0$ for all $n\neq 0.$ 
\end{proposition}
\begin{proof} The statement can be shown by an induction on the number of orbits. Denote by $j:G/H\cong\mathcal{O}\hookrightarrow X$ and $i:Z=X\backslash \mathcal{O}\to X$ the inclusion of an open orbit $\mathcal{O}$ and its closed complement $Z.$ Then the localisation triangle induces an exact sequence
\begin{center}
\begin{tikzcd}[column sep= 0.2cm]
 \Hom_{\DM_G(Z)}(i^*M,i^!N[n]) \arrow[r] \pgfmatrixnextcell \Hom_{\DM_G(X)}(M,N[n])\arrow[r]\pgfmatrixnextcell
 \Hom_{\DM_G(\mathcal{O})}(j^*M,j^!N[n]).
	\end{tikzcd}
\end{center}
Since $i^*M$ and $i^!N$ are $*$- and $!$-pointwise pure Tate, respectively, the first term of the sequence vanishes by induction. The last term vanishes since by assumption $j^*M$ and $j^!N$ correspond to objects in $\DTM_H(k,\Q)^{w=0}$ via the induction equivalence \eqref{eq:inductionequivalence} and thus have no non-trivial extension by Proposition~\ref{prop:extofpoint}. See \cite[Corollary II.4.19]{soergel_equivariant_2018} for a similar proof.
\end{proof}
Restricted to pointwise mixed Tate objects $\Real_\ell$ is a degrading functor after passing to $\Q_\ell$-coefficients.
\begin{proposition}\label{prop:gradedhompure}  Assume that the $G$-action on $X$ has finitely many orbits. Let $\ell\neq p$ be a prime and $M,N\in \DM_G(X).$  Then the natural isomorphisms $\Real_\ell\circ(1)\to\Real_\ell,$ see \eqref{eq:realforgetsgrading}, induces isomorphisms
$$\bigoplus_{n\in\Z}\Hom_{\DM_G(k,\Q_\ell)}(M,N(n))\stackrel{\sim}{\to} \Hom_{\DerG_G(k,\Q_\ell)}(\Real_\ell(M),\Real_\ell(N))$$ 
 if $M,N$  are $*$- and $!$-pointwise mixed Tate, respectively, and
$$\Hom_{\DM_G(X,\Q_\ell)}(M,N)\stackrel{\sim}{\to} \Hom_{\DerG_G(X,\Q_\ell)}(\Real_\ell(M),\Real_\ell(N))$$
if $M,N$ are $*$- and $!$-pointwise pure Tate, respectively.
\end{proposition}
\begin{proof}
As in the proof of Proposition~\ref{prop:homvanishingpure} the statement can be reduced to the case of a point where it is the same as Proposition~\ref{prop:puretateonpointgrading}.
\end{proof}
Pointwise pure Tate objects can be obtained from pushforwards along proper maps whose fibers have pure Tate motives.
\begin{proposition}\label{prop:puretatefiberimpliespuretate}
Let $\mu: M\to N$ be a $G$-equivariant proper map. Assume that $M$ is smooth and that the motives of the fibers of $\mu$ are pure Tate, 
$$\Mot(\mu^{-1}(\{x\})\in \DTM(k,\Q)^{w=0}.$$
Then the object $\mu_!(\Q_M)\in \DM_G(N,\Q)$ is pointwise pure Tate.
\end{proposition}
\begin{proof}
We first show that $\mu_!(\Q_M)$ is $*$-pointwise pure Tate. 
Let $i_x:\poi\to N$ be the inclusion of a point $x\in X.$ We have to show that
$$i_x^*\For\mu_!(\Q_M)\in\DTM(k,\Q)^{w=0}.$$
By applying base change with respect to the Cartesian diagram
\begin{center}
\begin{tikzcd}
\mu^{-1}(x) \arrow[r, "l"] \arrow[d, "\mu'"] & M \arrow[d, "\mu"] \\
\{x\} \arrow[r, "i_x"]                       & N                 
\end{tikzcd}
\end{center}
and the fact that $\For$ commutes with the six operations, we have
$$i_x^*\For\mu_!(\Q_M)=\mu'_!l^*\Q_M=\mu'_!\Q_{\mu^{-1}(x)}\in \DM(k,\Q).$$
Now $\mu'_!\Q_{\mu^{-1}(x)}$ is pure Tate since it is Verdier dual to the motive $\Mot^c(\mu^{-1}(\{x\})=\Mot(\mu^{-1}(\{x\})$
and Verdier duality preserves pure Tate motives.

That $\mu_!(\Q_M)=\mu_*(\Q_M)$ is $!$-pointwise pure Tate follows by using Verdier dual arguments.
\end{proof}

\section{Motivic Springer theory}\label{sec:motivicspringertheory}
In this section, we introduce the general setup and definitions of Springer motives and the motivic extension algebra. Moreover, we introduce the local purity and finiteness conditions (PT) and (FO).
We then combine the results of Sections~\ref{sec:weightstructuresandtilting} and~\ref{sec:formalismofequivariantmotivicsheaves} to obtain our main formality results for Springer motives.
\subsection{The setup}\label{sec:generalsetup}
Recall that all varieties are over $k=\overline{\F}_p.$
Let $G$ be a linear algebraic group. Let $\mu_i: \widetilde{\mathcal{N}}_i\to \N$ be a collection of $G$-equivariant proper maps where each $\M_i$ is smooth and connected.
Consider the collections of objects
\begin{align*}
    \Tt^{Spr}=\{\mu_{i,!}(\Q_{\M_i})\in \DM_G(\N)\} \text{ and }
    \widehat{\Tt}^{Spr}=\bigcup_{n\in \Z}\Tt^{Spr}(n)[2n].
\end{align*}
\begin{definition}\label{def:springermotives}
The triangulated category of \emph{Springer motives} $\DM^{Spr}_G(\N,\Q)$ is the full subcategory of $\DM_G(\N)$ generated by $\widehat{\Tt}^{Spr}$
with respect to isomorphism, direct summands and triangles,
$$\DM^{Spr}_G(\N,\Q)=\langle\widehat{\Tt}^{Spr}\rangle_{\cong, \inplus,\Delta}\subset \DM_G(\N).$$
\end{definition}
We define two conditions, that are assumed in some of the later results.
\begin{enumerate}
    \item[(PT)] \textbf{Pure Tate.} For each $x\in \N,$ the motive $\Mot(\mu_i^{-1}(\{x\})$ is pure Tate.
    \item[(FO)] \textbf{Finite Number of Orbits.} The images $\mu_i(\M_i)\subset \N$ have finitely many $G$-orbits.
\end{enumerate}
\begin{remark}
Condition (FO) allows for simple induction arguments. In many settings it can be weakened such that all arguments still work. For instance, there is a quite straightforward adaption to ind-varieties with possibly infinitely many orbits.
\end{remark}
\subsection{Motivic extension algebras}
We define now the motivic and $\ell$-adic extension algebras.
\begin{definition}\label{def:motivicextensionalgebra}
The \emph{motivic extension algebra} $E$ is the $\Z$-graded locally unital algebra $$E=\End^\bullet_{\DM_G(\N,\Q)}(\Tt^{Spr})$$ which has a grading induced by the autoequivalence $\langle 1\rangle=(1)[2],$ see \eqref{eq:gradedend}. More explicitly, for $n\in \Z$ the $n$-th graded part of $E$ is
$$E^n=\bigoplus_{i,j}\Hom_{\DM_G(\N,\Q)}(\mu_{i,!}(\Q_{\M_i}),\mu_{j,!}(\Q_{\M_j})(n)[2n]).$$
For every prime $\ell\neq 0$ we define the $\ell$-adic extension  algebra $E^{\acute{e}t}_\ell$ in the same way, replacing $\DM_G(\N,\Q)$ by $\DerG_G(\N,\Q_\ell).$
\end{definition}
The realisation functor induces a morphism of algebras
$\Real_\ell: E\to E^{\acute{e}t}_\ell$.
The motivic extension algebra can be understood in terms of $G$-equivariant Chow groups of the \emph{Steinberg varieties} $\Stein_{i,j}=\M_i\times_\N\M_j$ equipped with a convolution product. 
\begin{proposition}
Let  $d_j=\dim \M_j.$ There is a natural isomorphism
$$\Hom_{\DM_G(\N,\Q)}(\mu_{i,!}(\Q_{\M_i}),\mu_{j,!}(\Q_{\M_j})(n)[2n])=\Chow_{d_j-n}^G(Z_{i,j})_\Q.$$
\end{proposition}
\begin{proof} Let $Z=Z_{i,j}$ with projections $\pi_{i,j}:Z\to \M_{i,j}.$ For a variety $X$ denote by $\fin_X:X\to \poi$ the structure map. Then by using various adjunctions, base change and $\fin_{\M_j}^*=\fin_{\M_j}^!(-d_j)[-2d_j]$ since $\M_j$ is smooth, we get
\begin{align*}
    \Hom_{\DM_G(\N,\Q)}(&\mu_{i,!}(\Q_{\M_i}), \mu_{j,!}(\Q_{\M_j})(n)[2n])\\
    &\cong\Hom_{\DM_G(\M_i,\Q)}(\Q_{\M_i},\mu_i^!\mu_{j,*}(\Q_{\M_j})(n)[2n])\\
    &\cong\Hom_{\DM_G(\M_i,\Q)}(\fin_{\M_i}^*\Q,\pi_{i,*}\pi_{j}^!\fin_{\M_j}^*\Q(n)[2n])\\
    &\cong\Hom_{\DM_G(k,\Q)}(\Q,\fin_{\M_i,*}\pi_{i,*}\pi_{j}^!\fin_{\M_j}^!\Q(n-d_j)[2(n-d_j)])\\
    &\cong\Hom_{\DM_G(k,\Q)}(\Q,\fin_{Z,*}\fin_{Z}^!\Q(n-d_j)[2(n-d_j)])\\
    &\cong\Hom_{\DM_G(k,\Q)}(\Q,\Mot^c(Z)(n-d_j)[2(n-d_j)]).
\end{align*}
Now the last term is isomorphic to $\Chow^G_{d_j-n}(Z)_\Q$ by \cite[Theorem 5.3.14]{kelly_voevodsky_2017}.
\end{proof}
\begin{remark}
Since $Z_{i,j}$ is not necessarily equidimensional we need to work with Chow groups indexed by the \emph{dimension} of cycles here.
\end{remark}
The \emph{convolution} of two cycles
$\alpha\in \Chow^G_{d_j-n}(Z_{i,j})_\Q$ and $\beta\in \Chow^G_{d_k-m}(Z_{j,k})_\Q$ is given by the formula
\begin{align}
    \alpha\star\beta=p_*\delta^!(\alpha\times\beta)\in \Chow^G_{d_k-n-m}(Z_{i,k})_\Q \label{eq:convolution}
\end{align}
where $\alpha\times\beta$ is the exterior product and 
\begin{align*}
    \delta&: \M_i\times_\N\M_j\times_\N\M_k\to \M_i\times_\N\M_j\times\M_j\times_\N\M_k=Z_{i,j}\times Z_{j,k}\text{ and}\\
    p&: \M_i\times_\N\M_j\times_\N\M_k\to \M_i\times_\N\M_k = Z_{i,k}
\end{align*}
are the diagonal and projection maps.

Convolution of cycles and composition of morphisms are compatible in the obvious way. In the non-equivariant case this is proven and discussed in detail in \cite{fangzhou_borelmoore_2016}. The proof can be adapted to the equivariant case by using that both equivariant motivic sheaves and equivariant Chow groups are defined in terms of their non-equivariant versions of approximations of the Borel construction. We will not present the details here.
\begin{corollary}\label{cor:extensionalgebraaschowgroups} There is an isomorphism of graded algebras
\begin{align}\label{eq:extensionalgebraaschowgroups}
    E^\bullet\cong\bigoplus_{i,j}\Chow_{d_j-\bullet}^G(Z_{i,j})_\Q
\end{align}
\end{corollary}
Similarly, the $\ell$-adic extension algebra can be described in terms of $\ell$-adic Borel--Moore homology, see \cite[Section 8.6]{chriss_representation_2010}.
\begin{proposition}\label{prop:elladicextensionalgebraasborelmoore} There is an isomorphism of graded algebras
\begin{align}\label{eq:extensionalgebraborelmoore}
    (E^{\acute{e}t}_\ell)^\bullet\cong\bigoplus_{i,j}H^{BM,G}_{2(d_j-\bullet)}(Z_{i,j},\Q_\ell(d_j-\bullet)).
\end{align}
\end{proposition}
\begin{remark}
The above discussion should be a shadow of the following \textbf{conjectural} general theory.  There should be a Chow weight structure on the category $\DM_G^{gm}(\N)$ similarly to the non-equivariant case, see Example~\ref{exp:exampleweightstructures}(3). The heart of this Chow weight structure should be equivalent to a category of \emph{equivariant relative Chow motives} $\operatorname{Chow}_G(\N,\Q)$ in which the composition of morphisms is defined via convolution as in \eqref{eq:convolution}. Since by assumption $\M_i$ is smooth and $\mu_i$ is projective the motive $\mu_!(\Q_{M_i})$ should be in the heart and correspond to the relative \emph{Borel--Moore} motive $M^{BM}(\M_i/\N)$ in the category $\operatorname{Chow}_G(N,\Q).$

In the non-equivariant case this is shown to be true by Fangzhou \cite{fangzhou_borelmoore_2016}. To define a weight structure in the equivariant case, one would need appropriate $G$-equivariant resolution of singularities or alterations, see \cite[Remark II.4.15]{soergel_equivariant_2018}.

In this article we get around this problem by defining a weight structure on the subcategory $\DM^{Spr}_G(\N,\Q)$ by brute force using the conditions (PT) and (FO). 
\end{remark}
\subsection{Main results}
We will now combine the formalism of weight structures and weight complex functors from Section~\ref{sec:weightstructuresandtilting} with our results on pointwise pure Tate motives from Section~\ref{sec:pointwisetate} to obtain the following main result.
\begin{theorem}\label{thm:main}
Assume that (PT) and (FO) hold. Then there is an equivalence of categories
\begin{center}
\begin{tikzcd}
\DM^{Spr}_G(\N,\Q) \arrow[r, "\sim"] & {\DperfZ(E)}.
\end{tikzcd}
\end{center}
Moreover, for all primes $\ell\neq p$ the $\ell$-adic realisation functor $\Real_\ell$ gives an isomorphism $E\otimes_\Q\Q_\ell\cong E^{\acute{e}t}_\ell$ and acts as a degrading functor with respect to the Tate-twist $(1)$ in the sense of \cite{beilinson_koszul_1996}
\begin{center}
\begin{tikzcd}
\DM^{Spr}_G(\N,\Q_\ell) \arrow[r, "\Real_\ell"]&
{\D^{Spr}_G(\N,\Q_\ell).}                                   
\end{tikzcd}
\end{center}
\end{theorem}
The following is the most important ingredient in order to prove Theorem~\ref{thm:main}.
\begin{proposition}\label{prop:springerobjectstilting} Assume that the conditions (PT) and (FO) are fulfilled. Then the collection of objects $\widehat{\Tt}^{Spr}$ in $\DM_G(\N)$ is tilting.
\end{proposition}
\begin{proof}
The condition (PT) implies that all objects in $\widehat{\Tt}^{Spr}$ are pointwise pure Tate by Proposition~\ref{prop:puretatefiberimpliespuretate}. 
Now, let $M=\mu_{i,!}(\Q_{\M_i})(k)[2k]$ and $N=\mu_{j,!}(\Q_{\M_j})(l)[2l].$
The subvariety $k_{i,j}: \N_{i,j}=\mu_i(\M_i)\cup\mu_j(\M_j)\hookrightarrow \N$ is closed since $\mu_i$ and $\mu_j$ are proper. Since $M$ and $N$ are supported on $\N_{i,j}$ there is an equality
\begin{equation}\label{eq:proofofspringerobjectstilting}
    \Hom_{\DM_G(N,\Q)}(M,N[n])=\Hom_{\DM_G(\N_{i,j},\Q)}(k_{i,j}^*M,k_{i,j}^!N[n]).
\end{equation}
The condition (FO) ensures that there are only finitely many $G$-orbits in $\N_{i,j}.$ Moreover $k_{i,j}^*M$ and $k_{i,j}^!N$ are $*$- and $!$-pointwise pure Tate. Hence by Proposition~\ref{prop:homvanishingpure} the right hand side of \eqref{eq:proofofspringerobjectstilting} vanishes for $n\neq 0$ and the statement follows.
\end{proof}
By combining Propositions~\ref{prop:springerobjectstilting} and~\ref{prop:weightcomplexfunctor} we can define a weight structure $w$ on 
$\DM^{Spr}_G(\N,\Q)$ with heart
$\langle\widehat{\Tt}^{Spr}\rangle_{\cong, \inplus,\oplus}.$ Now, the first statement of Theorem~\ref{thm:main} follows from Corollary~\ref{cor:triangulatedcatequalsmodulecat}. The remaining statements are implied by the following:
\begin{proposition} Assume that the conditions (PT) and (FO) are fulfilled.  Let $M,N\in \DTM^{Spr}_G(\N)$ and $\ell\neq p$ be a prime. The natural isomorphisms $\Real_\ell\circ(1)\to\Real_\ell$ from \eqref{eq:realisationfunctor} induce  isomorphisms
$$\bigoplus_{n\in\Z}\Hom_{\DM_G(N,\Q_\ell)}(M,N(n))\stackrel{\sim}{\to} \Hom_{\DerG_G(\N,\Q_\ell)}(\Real_\ell(M),\Real_\ell(N))$$ and for $M,N\in\DTM^{Spr}_G(\N)^{w=0}$ isomorphisms $$\Hom_{\DM_G(N,\Q_\ell)}(M,N)\stackrel{\sim}{\to}\Hom_{\DerG_G(\N,\Q_\ell)}(\Real_\ell(M),\Real_\ell(N)).$$
\end{proposition}
\begin{proof}
Let $M,N\in\DTM^{Spr}_G(\N).$ All objects in $\widehat{\Tt}^{Spr}$ are supported on a closed subset of $\N$ consisting of finitely many $G$-orbits by condition (FO) and are pointwise pure Tate using condition (PT) and Proposition~\ref{prop:puretatefiberimpliespuretate}. Now $M$ and $N$ are constructed from the objects in $\widehat{\Tt}^{Spr}$ by a finite combination of taking direct summands, finite direct sums and triangles. Hence $M$ and $N$ are also supported on a closed subset of $\N$ consisting finitely many $G$-orbits and are pointwise mixed Tate. Now the statement follows from Proposition~\ref{prop:gradedhompure}. 
\end{proof}
\section{Springer resolution and affine Hecke algebras} \label{sec:heckealgebra}
We now apply Theorem~\ref{thm:main} to the special case of the \emph{Springer resolution}. 
Let $G$ be a reductive algebraic group over $\overline{\F}_p.$ Denote by $\Nnil\subset\operatorname{Lie}(G)$ the nilpotent cone, by $\mu:\M\to\Nnil$ the Springer resolution and by $Z=\M\times_{\Nnil}\M$ the Steinberg variety. There is an additional dilation action of $\mathbb{G}_m$ on $\M$ and $\Nnil,$ and we will consider the group $A=G\times \mathbb{G}_m$ or $A=G.$

Moreover, assume that 
\begin{enumerate}
    \item $p$ is a good prime for every classical group appearing as a constituent in $G$ and
    \item $p > 3(h + 1)$, where $h$ denotes the maximum of all Coxeter numbers of exceptional constituents in $G.$
\end{enumerate}

\begin{lemma}
The conditions (PT) and (FO) hold in this setup.
\end{lemma}
\begin{proof}
Using the conditions on $p$ 
\cite[Theorem 1.1]{eberhardt_springer_2021} shows that the motive $\Mot(\mu^{-1}(\{x\}))$ of a Springer fiber is pure Tate which implies (PT). Moreover, there are only finitely many $G$-orbits in $\Nnil,$ see \cite{carter_finite_1993}, which shows that condition (FO) holds.
\end{proof}

The motivic extension algebra $E$ for the action of $A=G\times\mathbb{G}_m$ can be identified with Lusztig's \emph{graded affine Hecke algebra} associated to $G$, see \cite{lusztig_cuspidal_1988} and \cite{lusztig_affine_1989},
$$E\cong(\Chow_{A}^\bullet(Z)_\Q,\star)\cong\overline{\mathbb{H}}(G).$$
We can now use Theorem~\ref{thm:main} to recover \cite[Corollary 1.4]{eberhardt_springer_2021}.
\begin{theorem}\label{thm:heckealgebra}
There is an equivalence of categories
$$\DM^{Spr}_{A}(\Nnil,\Q)\cong \DperfZ(\overline{\mathbb{H}}(G)).$$
\end{theorem}
For $A=G$ a similar result holds, where $E=S\rtimes\Q[W]$ is the semidirect product of the symmetric algebra $S=\operatorname{Sym}(X(T)_\Q)$ of the rationalized character lattice of a maximal torus $T$ in $G$ and the group algebra of the Weyl group $W=N_G(T)/T.$

\begin{remark}
It would be very desirable to obtain a similar result for the affine Hecke algebra $\mathbb{H}(G).$ Let $A=G\times\mathbb{G}_m.$ The affine Hecke algebra arises as the $A$-equivariant $K$-theory of the Steinberg variety
$\mathbb{H}(G)=K_0^{A}(Z)_\Q.$
Lusztig's graded affine Hecke algebra arises from a completion process from the affine Hecke algebra. 
The Atiyah-Segal completion theorem and the Chern character map identifies
$$\mathbb{H}(G)_I^{\wedge}=(K_0^{A}(Z)_\Q)_I^{\wedge}\cong K_0(EA\times_{A}Z)_\Q\cong\prod_i\Chow^i_{A}(Z)_\Q\supset \bigoplus_i\Chow^i_{A}(Z)_\Q=\overline{\mathbb{H}}(G),$$
the graded affine Hecke algebra as (a subspace of) the completion of the affine Hecke algebra at the augmentation ideal $I$ of the representation ring of $A.$ So, \textbf{conjecturally}, there should be an equivalence between the category of \emph{Springer $K$-motives} on the nilpotent cone
$$\operatorname{DK}^{Spr}_A(\Nnil,\Q)\cong \Dperf(\mathbb{H}(G))$$
and the perfect derived category of the affine Hecke algebra. Here, $\operatorname{DK}_A(\Nnil,\Q)$ denotes a (yet to be defined) category of equivariant $K$-motives which is an equivariant version of the $K$-motives defined by the first author in \cite{eberhardt_k-motives_2019}.

This conjecture highlights another advantage of motivic sheaves. Namely, one can construct categories of motivic sheaves for generalized cohomology theories such as $K$-theory.
\end{remark}

\section{Quiver Hecke (KLR) and quiver Schur algebras}\label{sec:quiverheckeandquiverschur}
We will now apply Theorem~\ref{thm:main} to quiver flag varieties and quiver Hecke and quiver Schur algebras in type $A$ and $\widetilde{A}.$
\subsection{Quiver flag varieties} We first recall some basic definitions and facts about quiver flag varietes. We refer to \cite{stroppel_quiver_2014} and \cite{przezdziecki_quiver_2019} for more details.
Consider a quiver $Q$ with finite sets of vertices $Q_0$ and arrows $Q_1$ and source and target maps $$s,t:Q_1\rightrightarrows Q_0.$$
Denote by $\Gamma=\Z_{\geq 0} Q_0$ and $\Gamma^+=\Gamma\backslash \{0\}$ the sets of (non-trivial) dimension vectors. The \emph{dimension vector} of a $Q_0$-graded $k$-vector space $V=(V_i)_{i\in Q_0}$ is $\dim(V)=(\dim(V_i))_{i\in Q_0}\in \Gamma.$ 
We denote by
$$\Rep(V)=\prod_{a\in Q_1}\Hom_k(V_{s(a)},V_{t(a)})$$
the vector space of \emph{quiver representations} with underlying $Q_0$-graded vector space~$V.$
We fix a dimension vector $\mathbf{d}\in \Gamma$ and let $V$ be the standard vector space with $\dim V=\mathbf{d}.$ We often abbreviate $\Rep(\mathbf{d})=\Rep(V).$ There is a a natural conjugation action on $\Rep(\mathbf{d})$ by the group
$$\GL(\mathbf{d})=\prod_{i\in Q_0}\GL(V_i)=\prod_{i\in Q_0}\GL_{\mathbf{d}_i}(k).$$

A \emph{composition} of a dimension vector $\mathbf{d}$ is a tuple 
$\underline{\mathbf{d}}=(\underline{\mathbf{d}}^j)\in (\Gamma^+)^{\ell_{\underline{\mathbf{d}}}}$ 
which sums to $\mathbf{d}.$ A composition $\underline{\mathbf{d}}$ is called \emph{complete} if each $\underline{\mathbf{d}}^j$ is a unit vector.
We write $\Comp(\mathbf{d})$ and $\Compf(\mathbf{d})$ for the set of (complete) compositions of $\mathbf{d}.$ A \emph{partial flag} $\underline{V}$ of $V$ of type $\underline{\mathbf{d}}$ is a sequence of $Q_0$-graded $k$-vector spaces
$$0=V^0\subset V^1\subset\dots\subset V^{\ell_{\underline{\mathbf{d}}}}=V$$
such that $\dim V^j/V^{j-1}=\underline{\mathbf{d}}^i.$ We denote the smooth projective variety of such flags by $\Fl(V,\underline{\mathbf{d}})=\Fl(\underline{\mathbf{d}}).$
The partial flag $\underline{V}$ is called \emph{strictly $\rho$-stable} for a quiver representation $\rho\in \Rep(\mathbf{d})$ if 
\begin{align}
    \rho_a(V^j_{s(a)})\subset V^{j-1}_{t(a)}\text{ for all } i=1,\dots,\ell_{\underline{\mathbf{d}}}\text{ and } a\in Q_1. \label{eq:strictlystable}
\end{align}
We can hence consider the variety
$$\QQ(\underline{\mathbf{d}})=\{(\rho,\underline{V})\in \Rep(\mathbf{d})\times \Fl(\underline{\mathbf{d}}) \mid \underline{V}\text{ is strictly $\rho$-stable}\}.$$
The variety $\QQ(\underline{\mathbf{d}})$ is smooth and has a diagonal action by $\GL(\mathbf{d}).$  Projection yields a $\GL(\mathbf{d})$-equivariant proper map
$$\mu_{\underline{\mathbf{d}}}:\QQ(\underline{\mathbf{d}})\to \Rep(\mathbf{d}).$$
For a representation $M=(V,\rho)$, the fiber $\mu_{\underline{\mathbf{d}}}^{-1}(\rho)$ is called a \emph{(partial) quiver flag variety} and denoted by
$$\Fl(M,\underline{\mathbf{d}})=\{\underline{V}\in \Fl(V,\underline{\mathbf{d}})\mid \underline {V}\text{ is strictly $\rho$-stable}\}.$$
\subsection{Conditions (PT) and (FO)}\label{sec:conditionsptandfo}
We want to apply the general setup of Springer motives from Section~\ref{sec:generalsetup} to the collection of maps $\mu_{\underline{\mathbf{d}}}:\QQ(\underline{\mathbf{d}})\to \Rep(\mathbf{d}).$ In particular, we are interested in settings where the condition (PT) and (FO) from Section~\ref{sec:generalsetup} are fulfilled. We restrict our attention to the following two cases.
\begin{enumerate}
    \item[($A$)] the quiver $Q$ is a \emph{Dynkin quiver} where the underlying graph is a Dynkin diagram of type $A.$
    \item[($\widetilde{A}$)] the quiver $Q$ is the \emph{cyclic quiver} with cyclic orientation, so the underlying graph is a Dynkin diagram of type $\widetilde{A}$.
\end{enumerate}
\begin{proposition}\label{prop:condictionsptandpointypea}
In case ($A$) and ($\widetilde{A}$) the condtions (PT) and (PO) hold.
\end{proposition}
\begin{proof}
In the case ($A$) there are only finitely many isomorphism classes of quiver representation with a fixed dimension vector by Gabriel's theorem. 
Hence there are only finitely many $\GL(\mathbf{d})$ orbits in $\Rep(\mathbf{d})$ and condition (FO) is fulfilled. 
The works of Cerulli-Irelli--Esposito--Franzen--Reineke \cite{cerulli_irelli_cell_2021} and Maksimau \cite{maksimau_flag_2019} show that type $A$ partial quiver flag varieties admit an affine pavings. 
This implies condition (PT) using Proposition~\ref{prop:affinepavingpuretate} and Proposition~\ref{prop:puretatefiberimpliespuretate}. 

In the case ($\widetilde{A}$) there can be infinitely many isomorphism classes of quiver representations with fixed dimension vector. 
However, any representation $\rho\in \Rep(\mathbf{d})$ in the image of $\mu_{\underline{\mathbf{d}}}$ fulfills \eqref{eq:strictlystable} for some flag $\underline{V}\in \Fl(\underline{\mathbf{d}}).$ 
This implies that $\rho$ is nilpotent, that is, there is some $n\in \Z_{\geq 0}$ such that $\rho_{a_1}\rho_{a_2}\dots\rho_{a_n}=0$ for any sequence of composable arrows $a_1,\dots,a_n\in Q_1.$ 
There are only finitely many isomorphism classes of nilpotent quiver representations with fixed dimension vector in this case, see Section~\ref{sec:cyclicquiver}. 
This implies condition (FO). 
Moreover, we will show in Section~\ref{sec:cyclicquiver} that the partial quiver flag varieties admit affine pavings which implies condition (PT).
\end{proof}
\begin{remark}
Quiver flag varieties in type $D$ also admit affine pavings by \cite{maksimau_flag_2019}, so our results also apply here. The same should hold in type $E.$ 
\end{remark}
\subsection{Affine pavings for quiver flag varieties in type $\widetilde{A}$}\label{sec:cyclicquiver}
Let $Q$ be the cyclic quiver on $n$ vertices. We label vertices and arrows $Q_0,Q_1=\Z/n$ such that $s(i)=i$ and $t(i)=i+1.$

Let $M=(V,\rho)$ be a representation of $M$ with dimension vector $\mathbf{d}=\dim V\in \Gamma$ and let $\underline{\mathbf{d}}\in\Comp(\mathbf{d})$ be a composition. he goal of this section is to prove the following theorem.
\begin{theorem}\label{thm:cylicquiverflagpaving}
If $M$ is a nilpotent representation of the cyclic quiver with cylic orientation $Q,$ then for all $\underline{\mathbf{d}}\in\Comp(\mathbf{d})$ the (partial) quiver flag variety $\Fl(M,\underline{\mathbf{d}})$ admits an affine paving.
\end{theorem}
For $\underline{\mathbf{d}}\in\Compf(\mathbf{d})$ this was shown by Sauter \cite{sauter_cell_2016}. We generalize her approach to work for arbitrary $\underline{\mathbf{d}}\in\Comp(\mathbf{d}).$

By \cite[Proposition 3.24]{schiffmann_lectures_2012} every  nilpotent representations $M=(V,\rho)$ of $Q$ is isomorphic to a direct sums of representations $E(i,l)$ defined as follows.
For $i\in \Z/n$ and $l\in \Z_\geq{0}$ we denote by $E(i,l)=(V,\rho)$ the representation with basis $e_{i-(l-1)},\dots,e_{i-1},e_i,$ such that $e_j\in V_j$ lives on the vertex $j\in\Z/n$ and $\rho_j(e_j)=e_{j+1}$ and $\rho_i(e_i)=0.$ The representation $E(i,l)$ is indecomposable and nilpotent with socle $\soc(E(i,l))=E(i,1)$ and radical filtration $\rad^n(E(i,l))=E(i,l-n).$

We discuss how to lift automorphims of the socle $\soc(M)$ to $M.$
Since $\rho$ restricted to $\soc(M)$ vanishes, we simply treat $\soc(M)$ as a $Q_0$-graded vector space. Choose $m>0$ such that $\rad^m(M)=0.$ Consider the flag obtained by intersecting the radical filtration of $M$ with the socle
$$I'=(\soc(M)\cap\rad^m(M)\subset\dots\subset\soc(M)\cap\rad(M)\subset \soc(M)).$$
Refine $I'$ to a complete flag $I$ of $\soc(M)$ and denote by $B\subset P\subset \GL(\soc(M))$ the stabilizers of $I$ and $I',$ respectively.
Denote the automorphism group of $M$ by $\Aut(M)\subset\GL(V).$ Restriction yields a natural morphism $\Res:\Aut(M)\to P.$
\begin{lemma} The morphism $\Res:\Aut(M)\to P$ has a section $\theta:P\to\Aut(M).$ 
\end{lemma}
\begin{proof}
We follow similar arguments to \cite[Lemma 1]{sauter_cell_2016}.

By the explicit description of nilpotent representations, $M$ is a direct sum of modules of the form $E(i,l).$ We can hence choose a basis of $\soc(M)$ by a choosing non-zero vectors in each $\soc(E(i,l)).$ With respect to this basis, $P$ is generated by elementary matrices and the proof can be reduced to the case $M=E(i,l_1)\oplus E(i,l_2).$ We denote the standard basis vectors of $E(i,l_1)$ and $E(i,l_2)$ by $e_j$ and $f_j,$ respectively. 

The socle of $\soc(M)=\soc(E(i,l_1))\oplus \soc(E(i,l_2))=ke_i\oplus kf_i$ is two-dimensional in degree $i.$ Assume that $l_1>l_2.$ Then the radical filtration of the socle is 
$$I'=(0\subset ke_i\subset ke_i\oplus kf_i).$$
Let $g\in P.$ For $j< l_2$ we define the action $\theta(g)$ on $ke_{i-j}\oplus kf_{i-j}$ via the natural isomorphism $ke_{i-j}\oplus kf_{i-j}\cong ke_i\oplus kf_i.$ For $j\geq l_2,$ we define the action $\theta(g)$ on $ke_{i-j}$ via the natural isomorphism $ke_{i-j}\cong ke_i.$ It can easily be checked that $\theta(g)\in \Aut(M)$ and $\Res\theta(g)=g.$ The case $l_1=l_2$ is proven similarly.
\end{proof}

Since $M=(V,\rho)$ is nilpotent, we have $\soc(M)=\ker(\rho).$ Hence, for every flag $\underline{V}=(0\subset V^1\subset\dots)$ that is strictly $\rho$-stable we have $V_1\subset \soc(M),$ see \eqref{eq:strictlystable}. We hence get a natural map
\begin{align}
    p:\Fl(M,\underline{\mathbf{d}})\to \Gr(\soc(M),\underline{\mathbf{d}}^1),\,\underline{V}\mapsto V_1\label{eqn:inductionforquiverflag}
\end{align}
from the partial quiver flag variety to the Grassmannian of subspaces of $\soc(M)$ with dimension vector $\underline{\mathbf{d}}^1=\dim V^1.$ We  construct an affine paving of $\Fl(M,\underline{\mathbf{d}})$ using $p$ inductively. We will show that the preimage under $p$ of a $B$-orbit in the Grassmannian has an affine paving and then deduce the claim.
\begin{proof}[Proof of Theorem~\ref{thm:cylicquiverflagpaving}]
Choose any $V'\in \Gr(\soc(M),\underline{\mathbf{d}}_1)$ that is stabilized by $B.$ Denote the stabilizer of $V'$ by $Q\subset \GL(\soc(M))$ and the respective Weyl groups by $W_Q\subset W.$ Denote by $W^Q$ the set of shortest coset representatives in $W/W_Q.$ 
Let $U\subset B$ be the unipotent radical, $U^-\subset G$ its opposite and $U_x=U\cap xU^-x^{-1}$ for $x\in W.$ Then $U_x\cong \mathbb{A}^{l(x)}$ is an affine space. 

The Grassmannian $\Gr(\soc(M),\underline{\mathbf{d}}_1)$ admits an affine paving by $B$-orbits
$$\Gr(\soc(M),\underline{\mathbf{d}}^1)=\biguplus_{w\in W^Q}\Gr(\soc(M),\underline{\mathbf{d}}^1)_w$$
such that for $w\in W^Q$ there is an isomorphism
$$t_w: U_w\to \Gr(\soc(M),\underline{\mathbf{d}}^1)_w, u\mapsto uwV'.$$
Taking preimages under the map $p$ from \eqref{eqn:inductionforquiverflag} yields a decomposition 
$$\Fl(M,\underline{d})=\biguplus_{w\in W^Q}\Fl(M,\underline{d})_w.$$
Denote by $\underline{\mathbf{d}}^\star=(\underline{\mathbf{d}}^2,\dots)\in \Comp(\mathbf{d}-\underline{\mathbf{d}}^1)$ the composition obtained by removing the first entry $\underline{\mathbf{d}}^1$ from $\underline{\mathbf{d}}.$ For $w\in W^Q$ consider the map
$$\alpha:U_w\times \Fl(M/wV',\underline{\mathbf{d}}^\star)\to \Fl(M,\underline{d})_w,\, (u,\underline{W})\mapsto \theta(u)(\underline{W}+wV')$$
where for a flag $\underline{W}=(0\subset W^1\subset\dots)$ of $M/wV'$ we denote the lift to a flag of $M$ by $\underline{W}+wV=(0\subset wV\subset W^1+wV\subset\dots).$ 
The map $\alpha$ is an isomorphism with inverse
$$\beta:\Fl(M,\underline{d})_w\to U_w\times \Fl(M/wV',\underline{\mathbf{d}}^\star),\, \underline{V}\mapsto(u(V^1), \theta(u(V^1)^{-1})(\underline{V})/wV')$$
where we write $u(V^1)=t_w^{-1}(V^1)$ and for a flag $\underline{V}=(0\subset V^1\subset\dots)$ of $M$ we denote by $\underline{V}/V^1=(0\subset V^2/V^1\subset \dots)$ the projection to a flag of $M/V^1.$

By induction, each $\Fl(M/wV',\underline{\mathbf{d}}^\star)$ admits an affine paving which implies that $\Fl(M,\underline{\mathbf{d}})$ does.
\end{proof}

\subsection{Quiver Hecke and quiver Schur algebras}
Assume that we are in the cases ($A$) or ($\widetilde{A}$) of Section~\ref{sec:conditionsptandfo}. For a fixed dimension vector $\mathbf{d}\in \Gamma$ we consider the collection of $\GL(\mathbf{d})$-equivariant maps 
$$\mu_{\underline{\mathbf{d}}}:\QQ(\underline{\mathbf{d}})\to \Rep(\mathbf{d})$$
where $\underline{\mathbf{d}}\in \Compf(\mathbf{d})$ ranges over all \emph{complete compositions}. 
Denote by $R_\mathbf{d}$ the \emph{quiver Hecke (KLR) algebra} associated to $Q$ and $\mathbf{d}$ as defined by Khovanov--Lauda \cite{khovanov_diagrammatic_2009} and Rouquier \cite{rouquier_2-kac-moody_2008}. 

\begin{theorem}\label{thm:quiverheckealgebra}
There is an equivalence of categories 
$$\DM^{Spr}_{\GL(\mathbf{d})}(\Rep(\mathbf{d}),\Q)\cong \DperfZ(R_\mathbf{d})$$ 
between Springer motives with respect to $\mu_{\underline{\mathbf{d}}}:\QQ(\underline{\mathbf{d}})\to \Rep(\mathbf{d})$ for complete compositions $\underline{\mathbf{d}}\in \Compf(\mathbf{d})$ and the perfect derived category of graded modules of the quiver Hecke algebra.
\end{theorem}
\begin{proof}
The motivic extension algebra $E$ can be identified with $R_\mathbf{d}.$ To see this, one can adapt the proof of Varagnolo--Vasserot \cite[Theorem 3.6]{varagnolo_canonical_2011} from the context of equivariant Borel--Moore homology to Chow groups. Their arguments apply unchanged making use of the fact that the partial quiver flag varieties admit affine pavings and hence their equivariant Borel--Moore homology and Chow groups coindice.
By Proposition~\ref{prop:condictionsptandpointypea} conditions (PT) and (PO) hold and the statement follows by Theorem~\ref{thm:main}.
\end{proof}

If we let $\underline{\mathbf{d}}\in \Comp(\mathbf{d})$ range over all compositions, the  motivic extension algebra $E$ can be identified with the \emph{quiver Schur algebra} $A_{\mathbf{d}}$ defined by Stroppel--Webster \cite{stroppel_quiver_2014} and we get:
\begin{theorem}\label{thm:quiverschuralgebra} There is an equivalence of categories
$$\DM^{Spr}_{\GL(\mathbf{d})}(\Rep(\mathbf{d}),\Q)\cong \DperfZ(A_\mathbf{d})$$
between Springer motives with respect to $\mu_{\underline{\mathbf{d}}}:\QQ(\underline{\mathbf{d}})\to \Rep(\mathbf{d})$ for compositions $\underline{\mathbf{d}}\in \Comp(\mathbf{d})$ and the perfect derived category of graded modules of the quiver Schur algebra.
\end{theorem}

\bibliographystyle{amsalpha} 
\bibliography{main}

\end{document}